\documentclass{amsproc}
\usepackage{amssymb}
\usepackage{mathrsfs}
\usepackage{hyperref}
\usepackage{accents}
\usepackage[all]{xy}

\newcommand{\Z}{\mathbb{Z}}
\newcommand{\C}{\mathbb{C}}
\newcommand{\bFp}{\bar{\mathbb{F}}_p}
\newcommand{\Qlb}{\bar{\mathbb{Q}}_{\ell}}
\newcommand{\bk}{\Bbbk}

\newcommand{\cO}{\mathcal{O}}

\newcommand{\Gm}{\mathbb{G}_{\mathrm{m}}}
\newcommand{\Ga}{\mathbb{G}_{\mathrm{a}}}
\newcommand{\SL}{\mathrm{SL}}
\newcommand{\bA}{\mathbb{A}}

\newcommand{\Gv}{\check G}
\newcommand{\Bv}{\check B}
\newcommand{\Uv}{\check U}
\newcommand{\Fl}{\mathcal{F}l}
\newcommand{\Gr}{\mathcal{G}r}
\newcommand{\bK}{{\mathbf{K}}}
\newcommand{\bO}{{\mathbf{O}}}
\newcommand{\IW}{{\mathrm{IW}}}
\newcommand{\cX}{\mathcal{X}}
\newcommand{\Perv}{\mathrm{Perv}}
\newcommand{\Parity}{\mathrm{Parity}}
\newcommand{\ext}{{\mathrm{ext}}}
\newcommand{\mix}{{\mathrm{mix}}}

\newcommand{\Db}{D^{\mathrm{b}}}
\newcommand{\Coh}{\mathrm{Coh}}
\newcommand{\Tilt}{\mathrm{Tilt}}
\newcommand{\ExCoh}{\mathrm{ExCoh}}
\newcommand{\PCoh}{\mathrm{PCoh}}

\newcommand{\fE}{\mathfrak{E}}
\newcommand{\fIC}{\mathfrak{IC}}

\newcommand{\hfT}{\hat{\mathfrak{T}}}

\newcommand{\bDelta}{\bar\Delta}
\newcommand{\bnabla}{\bar\nabla}
\newcommand{\hDelta}{\hat\Delta}
\newcommand{\hnabla}{\hat\nabla}

\newcommand{\cE}{\mathcal{E}}
\newcommand{\cF}{\mathcal{F}}

\newcommand{\cH}{\mathcal{H}}
\newcommand{\cK}{\mathcal{K}}

\newcommand{\cN}{\mathcal{N}}

\newcommand{\tcN}{{\widetilde{\mathcal{N}}}}
\newcommand{\reg}{{\mathrm{reg}}}

\newcommand{\fu}{\mathfrak{u}}
\newcommand{\fg}{\mathfrak{g}}

\newcommand{\bX}{\mathbf{X}}

\newcommand{\bXp}{\mathbf{X}^+}

\newcommand{\hrho}{\hat{\rho}}
\newcommand{\bXmin}{\mathbf{X}_{\mathrm{min}}}

\newcommand{\dom}{\mathsf{dom}}
\newcommand{\smm}{\bar{\mathsf{m}}}
\newcommand{\smp}{\accentset{+}{\mathsf{m}}}

\newcommand{\sA}{\mathscr{A}}

\newcommand{\la}{\langle}
\newcommand{\ra}{\rangle}
\newcommand{\Rep}{\mathrm{Rep}}
\newcommand{\Hom}{\mathrm{Hom}}
\newcommand{\Ext}{\mathrm{Ext}}
\newcommand{\uHom}{\underline{\mathrm{Hom}}}
\newcommand{\uEnd}{\underline{\mathrm{End}}}
\newcommand{\uExt}{\underline{\mathrm{Ext}}}
\newcommand{\RHom}{R\Hom}
\newcommand{\cRHom}{R\mathcal{H}\mathit{om}}

\DeclareMathOperator{\res}{res}
\DeclareMathOperator{\ind}{ind}

\DeclareMathOperator{\im}{im}
\DeclareMathOperator{\codim}{codim}
\DeclareMathOperator{\supp}{supp}
\DeclareMathOperator{\ch}{ch}
\newcommand{\SGD}{\mathbb{D}}
\newcommand{\pt}{\mathrm{pt}}
\newcommand{\lmod}{\textup{-}\mathrm{mod}}

\newcommand{\tl}{\triangleleft}

\newcommand{\simto}{\overset{\sim}{\to}}

\numberwithin{equation}{section}
\newtheorem{thm}{Theorem}[section]
\newtheorem{lem}[thm]{Lemma}
\newtheorem{prop}[thm]{Proposition}
\newtheorem{cor}[thm]{Corollary}
\newtheorem{conj}[thm]{Conjecture}
\newtheorem{ques}[thm]{Question}
\theoremstyle{definition}

\theoremstyle{remark}
\newtheorem{rmk}[thm]{Remark}

\newenvironment{proofsk}{\begin{proof}[Proof sketch]%
  }%
  {\end{proof}}
  
\newcommand{\nocontentsline}[3]{}
\newcommand{\tocless}[2]{\bgroup\let\addcontentsline=\nocontentsline#1{#2}\egroup}

\newcounter{exercise}
\title{Notes on exotic and perverse-coherent sheaves}

\author{Pramod N. Achar}
\address{Department of Mathematics\\
  Louisiana State University\\
  Baton Rouge, LA 70803\\
  U.S.A.}
\email{pramod@math.lsu.edu}

\dedicatory{Dedicated to David Vogan on his 60th birthday}

\thanks{This work was supported by NSF grant~DMS-1001594 and NSA grant~H98230-14-1-0117.}

\begin{document}

\begin{abstract}
Exotic sheaves are certain complexes of coherent sheaves on the cotangent bundle of the flag variety of a reductive group.  They are closely related to perverse-coherent sheaves on the nilpotent cone.  This expository article includes the definitions of these two categories, applications, and some structure theory, as well as detailed calculations for $\SL_2$.
\end{abstract}

\maketitle

Let $G$ be a reductive algebraic group.  Let $\cN$ be its nilpotent cone, and let $\tcN$ be the Springer resolution.  This article is concerned with two closely related categories of complexes of coherent sheaves: \emph{exotic sheaves} on $\tcN$, and \emph{perverse-coherent sheaves} on $\cN$.  These categories play key roles in Bezrukavnikov's proof of the Lusztig--Vogan bijection~\cite{bez:qes} and his study of quantum group cohomology~\cite{bez:ctm}; in the Bezrukavnikov--Mirkovi\'c proof of Lusztig's conjectures on modular representations of Lie algebras~\cite{bm:rssla}; in the proof of the Mirkovi\'c--Vilonen conjecture on torsion on the affine Grassmannian~\cite{ar:psag}; and in various derived equivalences related to local geometric Langlands duality, both in characteristic zero~\cite{ab:psaf, bez:psaf} and, more recently, in positive characteristic~\cite{ar:agsr, mr:etsps}.

In this expository article, we will introduce these categories via a relatively elementary approach.  But we will also see that they each admit numerous alternative descriptions, some of which play crucial roles in various applications.  The paper also touches on selected aspects of their structure theory, and it concludes with detailed calculations for $G = \SL_2$.  Some proofs are sketched, but many are omitted.  There are no new results in this paper.

My interest in these notions began when David Vogan pointed out to me, in 2001, the connection between the papers~\cite{bez:pc, bez:qes} and my own Ph.D. thesis~\cite{a:phd}.  It is a pleasure to thank David for suggesting a topic that has kept me interested and busy for the past decade and a half!

\setcounter{tocdepth}{2}
\tableofcontents

\section{Definitions and preliminaries}
\label{sect:defn}

\subsection{General notation and conventions}
\label{ss:notation}

Let $\bk$ be an algebraically closed field, and let $G$ be a connected reductive group over $\bk$.  We assume throughout that the following conditions hold:
\begin{itemize}
\item The characteristic of $\bk$ is zero or a JMW prime for $G$.
\item The derived subgroup of $G$ is simply connected.
\item The Lie algebra $\fg$ of $G$ admits a nondegenerate $G$-invariant bilinear form.
\end{itemize}
Here, a \emph{JMW prime} is a prime number that is good for $G$ and such that the main result of~\cite{jmw:pstm} holds for $G$.  That result, which is concerned with tilting $G$-modules under the geometric Satake equivalence, holds for quasisimple $G$ at least when $p$ satisfies the following bounds:
\[
\begin{array}{c|c|c|c}
\text{$A_n$, $B_n$, $D_n$, $E_6$, $F_4$, $G_2$} & C_n & E_7 & E_8 \\
\hline
\text{$p$ good for $G$} & p > n & p > 19 & p > 31
\end{array}
\]
The condition that $\fg$ admit a nondegenerate $G$-invariant bilinear form is satisfied for $\mathrm{GL}(n)$ in all characteristics, and for quasisimple, simply connected groups not of type $A$ in all good characteristics.

Fix a Borel subgroup $B \subset G$ and a maximal torus $T \subset B$.  Let $\bX$ be the character lattice of $T$, and let $\bXp \subset \bX$ be the set of dominant weights corresponding to the positive system \emph{opposite} to the roots of $B$.  (In other words, we think of $B$ as a ``negative'' Borel subgroup.)  Let $\fu \subset \fg$ be the Lie algebra of the unipotent radical of $B$.  Let $W$ be the Weyl group, and let $w_0 \in W$ be the longest element.  For $\lambda \in \bX$, let
\[
\delta_\lambda = \min \{ \ell(w) \mid w\lambda \in \bXp \}
\qquad\text{and}\qquad
\delta^*_\lambda = \delta_{w_0\lambda}.
\]
For any weight $\lambda \in \bX$, let $\dom(\lambda)$ be the unique dominant weight in its Weyl group orbit.  Two partial orders on $\bX$ will appear in this paper.  For $\lambda, \mu \in \bX$, we write
\begin{center}
\begin{tabular}{ll}
$\lambda \preceq \mu$ & if $\mu - \lambda$ is a sum of positive roots; \\
$\lambda \leq \mu$ & if $\dom(\lambda) \prec \dom (\mu)$, or else if $\dom(\lambda) = \dom(\mu)$ and $\lambda \preceq \mu$.
\end{tabular}
\end{center}

Let $\bXmin \subset \bXp$ be the set of minuscule weights.  Elements of $-\bXmin$ are called \emph{antiminuscule} weights.  For any $\lambda \in \bX$, there is a unique minuscule, resp.~antiminuscule, weight that differs from $\lambda$ by an element of the root lattice, denoted 
\[
\smp(\lambda),
\qquad\text{resp.}\qquad
\smm(\lambda).
\]
Note that $\smm(\lambda) = w_0\smp(\lambda)$.

Let $\Rep(G)$ and $\Rep(B)$ denote the categories of finite-dimensional algebraic $G$- and $B$-representations, respectively.  Let $\ind_B^G: \Rep(B) \to \Rep(G)$ and $\res^G_B: \Rep(G) \to \Rep(B)$ denote the induction and restriction functors.  For any $\lambda \in \bX$, let $\bk_\lambda \in \Rep(B)$ be the $1$-dimensional $B$-module of weight $\lambda$, and let
\[
H^i(\lambda) = R^i\ind_B^G \bk_\lambda.
\]
If $\lambda \in \bXp$, then $H^0(\lambda)$ is the dual Weyl module of highest weight $\lambda$, and the other $H^i(\lambda)$ vanish.  On the other hand, we put
\[
V(\lambda) = H^{\dim G/B}(w_0\lambda - 2\rho) \cong H^0(-w_0\lambda)^*.
\]
This is the Weyl module of highest weight $\lambda$.

Let $\cN \subset \fg$ be the nilpotent cone of $G$, and let $\tcN = G \times^B \fu$.  Any weight $\lambda \in \bX$ gives rise to a line bundle $\cO_\tcN(\lambda)$ on $\tcN$.  The Springer resolution is denoted by
\[
\pi: \tcN \to \cN.
\]

Let the multiplicative group $\Gm$ act on $\fg$ by $z \cdot x = z^{-2}x$, where $z \in \Gm$ and $x \in \fg$.  This action commutes with the adjoint action of $G$.  It restricts to an action on $\cN$ or on $\fu$; the latter gives rise to an induced action on $\tcN$.  We write $\Coh^{G \times \Gm}(\cN)$ for the category of $(G \times \Gm)$-equivariant coherent sheaves on $\cN$, and likewise for the other varieties. Recall that there is an ``induction equivalence''
\begin{equation}\label{eqn:ind-equiv}
\Coh^{B \times \Gm}(\fu) \cong \Coh^{G \times \Gm}(\tcN).
\end{equation}

The $\Gm$-action on $\fg$ corresponds to equipping the coordinate ring $\bk[\fg]$ with a grading in which the space of linear functions $\fg^* \subset \bk[\fg]$ is precisely the space of homogeneous elements of degree~$2$.  Thus, a $(G \times \Gm)$-equivariant coherent sheaf on $\fg$ is the same as a finitely-generated graded $\bk[\fg]$-module equipped with a compatible $G$-action.  Similar remarks apply to $\Coh^{G \times \Gm}(\cN)$ and $\Coh^{B \times \Gm}(\fu)$.  If $V = \bigoplus_{n \in \Z} V_n$ is a such a graded module (or more generally, a graded vector space), we define $V\la n\ra$ to be the graded module given by $(V\la m\ra)_n = V_{m+n}$. Given two graded vector modules $V, W$, we define a new graded vector space
\[
\uHom(V,W) = \bigoplus_{n \in \Z} \uHom(V,W)_n
\qquad\text{where}\qquad
\uHom(V,W)_n = \Hom(V,W\la n\ra).
\]
We extend the notation $\cF \mapsto \cF\la n\ra$ to $\Coh^{G \times \Gm}(\tcN)$ via the equivalence~\eqref{eqn:ind-equiv}.

\begin{rmk}\label{rmk:mr}
In a forthcoming paper~\cite{mr:etsps}, Mautner and Riche prove that all good primes are JMW primes.  That paper will also give new proofs of the foundational results in Sections~\ref{ss:exotic} and \ref{ss:mu-exact} below, placing them in the context of the affine braid group action on $\Db\Coh^{G \times \Gm}(\tcN)$ introduced in~\cite{br:abga}.  The main result of~\cite{mr:etsps} appears in this paper as Theorem~\ref{thm:excoh-gr}.
\end{rmk}

\subsection{Exotic sheaves}
\label{ss:exotic}

This subsection and the following one introduce the two $t$-structures we will study.  On a first reading, the descriptions given here are, unfortunately, rather opaque.  We will see some other approaches in Section~\ref{sect:apps}; for explicit examples, see Section~\ref{sect:sl2}.

For $\lambda \in \bX$, let $\Db\Coh^{G \times \Gm}(\tcN)_{\le \lambda}$ be the full triangulated subcategory of $\Db\Coh^{G \times \Gm}(\tcN)$ generated by line bundles $\cO_\tcN(\mu)\la m\ra$ with $\mu \le \lambda$ and $m \in \Z$.  The category $\Db\Coh^{G \times \Gm}(\tcN)_{< \lambda}$ is defined similarly.

\begin{prop}\label{prop:mut}
For any $\lambda \in \bX$, there are objects $\hDelta_\lambda, \hnabla_\lambda \in \Db\Coh^{G \times \Gm}(\tcN)$ that are uniquely determined by the following properties:  there are canonical distinguished triangles
\begin{gather*}
\hDelta_\lambda \to \cO_\tcN(\lambda)\la \delta_\lambda\ra \to \cK_\lambda \to
\qquad\text{and}\qquad
\cK'_\lambda \to \cO_\tcN(\lambda)\la \delta_\lambda\ra \to \hnabla_\lambda \to
\end{gather*}
such that $\cK_\lambda \in \Db\Coh^{G \times \Gm}(\tcN)_{< \lambda}$, $\cK'_\lambda \in \Db\Coh^{G \times \Gm}(\tcN)_{< w_0\dom(\lambda)}$, and
\[
\uHom^\bullet(\hDelta_\lambda,\cF) = \uHom^\bullet(\cF, \hnabla_\lambda) = 0 \qquad\text{for all $\cF \in \Db\Coh^{G \times \Gm}(\tcN)_{<\lambda}$.}
\]
\end{prop}
\begin{proofsk}
This is mostly a consequence of the machinery of ``mutation of exceptional sets,'' discussed in~\cite[\S2.1 and~\S2.3]{bez:ctm}.  The general results of~\cite[\S2.1.4]{bez:ctm} might at first suggest that $\cK_\lambda$ and $\cK'_\lambda$ both lie just in $\Db\Coh^{G \times \Gm}(\tcN)_{< \lambda}$.  To obtain the stronger constraint on $\cK'_\lambda$, one first shows that line bundles already have a strong $\Hom$-vanishing property with respect to $\preceq$:
\begin{equation}\label{eqn:line-excep}
\uHom^\bullet(\cO_\tcN(\mu), \cO_\tcN(\lambda)) = 0 \qquad\text{if $\mu \not\succeq \lambda$.}
\end{equation}
The result is obtained by combining this with a study of the $\hnabla_\lambda$ based on Proposition~\ref{prop:braid} below.
\end{proofsk}

The preceding proposition says, in part, that the $\hnabla_\lambda$ constitute a ``graded exceptional set'' in the sense of~\cite[\S2.3]{bez:ctm}.  In the extreme cases of dominant or antidominant weights, equation~\eqref{eqn:line-excep} actually implies that
\begin{equation}\label{eqn:excoh-dom}
\hnabla_\lambda \cong \cO_\tcN(\lambda)
\qquad\text{and}\qquad
\hDelta_{w_0\lambda} \cong \cO_\tcN(w_0\lambda)\la \delta_{w_0\lambda}\ra
\qquad\text{for $\lambda \in \bXp$.}
\end{equation}

The proposition also implies that the composition $\hDelta_\lambda \to \cO_\tcN(\lambda)\la \delta_\lambda\ra \to \hnabla_\lambda$ is nonzero for all $\lambda \in \bX$.  This is the morphism $\hDelta_\lambda \to \hnabla_\lambda$ appearing in the following statement.

\begin{thm}\label{thm:excoh-defn}
There is a unique $t$-structure on $\Db\Coh^{G \times \Gm}(\tcN)$ whose heart
\[
\ExCoh^{G \times \Gm}(\tcN) \subset \Db\Coh^{G \times \Gm}(\tcN)
\]
is stable under $\la 1\ra$ and contains $\hDelta_\lambda$ and $\hnabla_\lambda$ for all $\lambda \in \bX$. In this category, every object has finite length.  The objects
\[
\fE_\lambda\la n\ra := \im(\hDelta_\lambda\la n\ra \to \hnabla_\lambda\la n\ra)
\]
are simple and pairwise nonisomorphic, and every simple object is isomorphic to one of these.
\end{thm}

This $t$-structure is called the \emph{exotic $t$-structure} on $\Db\Coh^{G \times \Gm}(\tcN)$.  The $\hDelta_\lambda$ and $\hnabla_\lambda$ are called \emph{standard} and \emph{costandard} objects, respectively.  Further properties of these objects will be discussed in Section~\ref{ss:prop-strat}.

\begin{proofsk}
The existence of the $t$-structure follows from~\cite[Proposition~4]{bez:ctm}, which describes a general mechanism for constructing $t$-structures from exceptional sets.  However, that general mechanism does not guarantee that the $\hDelta_\lambda$ and $\hnabla_\lambda$ lie in the heart.  One approach to showing that they do lie in the heart (see~\cite{ar:agsr}) is to deduce it from the derived equivalences that we will see in Sections~\ref{ss:whittaker} (for $\bk = \C$) or~\ref{ss:aff-grass} (in general).
\end{proofsk}

\subsection{Perverse-coherent sheaves}
\label{ss:pcoh}

For any $\lambda \in \bX$, let
\[
A_\lambda = \pi_*\cO_\tcN(\lambda).
\]
For $\lambda \in \bXp$, we define $\Db\Coh^{G \times \Gm}(\cN)_{\le \lambda}$ to be the full triangulated subcategory of $\Db\Coh^{G \times \Gm}(\cN)$ generated by the objects $A_\mu\la m\ra$ with $\mu \in \bXp$, $\mu \le \lambda$, and $m \in \Z$.  The category $\Db\Coh^{G \times \Gm}(\cN)_{< \lambda}$ is defined similarly.  Note that these categories are defined only for dominant weights, unlike in the exotic case.

One basic property of the $A_\lambda$, often called \emph{Andersen--Jantzen sheaves}, is that if $s_\alpha \in W$ is a simple reflection such that $s_\alpha\lambda \prec \lambda$, then there is a natural map
\begin{equation}\label{eqn:aj-refl}
A_{s_\alpha\lambda} \to A_\lambda\la -2\ra
\end{equation}
whose cone lies in $\Db\Coh^{G \times \Gm}(\cN)_{< \dom(\lambda)}$~\cite[Lemma~5.3]{a:pcsnc}.

For a dominant weight $\lambda \in \bXp$, we introduce the following additional notation:
\begin{equation}\label{eqn:bnabla-defn}
\bDelta_\lambda = A_{w_0\lambda}\la \delta^*_\lambda\ra,
\qquad
\bnabla_\lambda = A_\lambda\la -\delta^*_\lambda\ra.
\end{equation}
It follows from~\eqref{eqn:aj-refl} that there is a natural nonzero map
\[
\bDelta_\lambda \to \bnabla_\lambda
\]
whose cone lies in $\Db\Coh^{G \times \Gm}(\cN)_{< \lambda}$.  This morphism appears in the following statement.

\begin{thm}\label{thm:pcoh-defn}
There is a unique $t$-structure on $\Db\Coh^{G \times \Gm}(\cN)$ whose heart
\[
\PCoh^{G \times \Gm}(\cN) \subset \Db\Coh^{G \times \Gm}(\cN)
\]
is stable under $\la 1\ra$ and contains $\bDelta_\lambda$ and $\bnabla_\lambda$ for all $\lambda \in \bXp$.  In this category, every object has finite length.  The objects
\[
\fIC_\lambda\la n\ra := \im(\bDelta_\lambda\la n\ra \to \bnabla_\lambda\la n\ra)
\]
are simple and pairwise nonisomorphic, and every simple object is isomorphic to one of these.
\end{thm}
This $t$-structure is called the \emph{perverse-coherent $t$-structure} on $\Db\Coh^{G \times \Gm}(\cN)$.  The justification for this terminology will be discussed in Section~\ref{ss:local}. 
\begin{proofsk}
The first step is to show that the $\bDelta_\lambda$ and $\bnabla_\lambda$ satisfy $\Ext$-vanishing properties similar to (but somewhat weaker than) those in Proposition~\ref{prop:mut}.  (To be precise, the $\bnabla_\lambda$ constitute a ``graded quasiexceptional set,'' but not an exceptional set.  See~\cite[\S2.2 and Remark~8]{bez:qes}.)  The existence of the $t$-structure follows from a general mechanism, as in Theorem~\ref{thm:excoh-defn}.  In this case, the fact that the $\bDelta_\lambda$ and the $\bnabla_\lambda$ lie in the heart can be checked by direct computation; see~\cite[Lemma~9]{bez:qes} or~\cite[Lemma~5.2]{a:pcsnc}.
\end{proofsk}

\subsection{The relationship between exotic and perverse-coherent sheaves}
\label{ss:mu-exact}

In this subsection, we briefly outline a proof of the $t$-exactness of $\pi_*$, following~\cite[\S\S2.3--2.4]{bez:ctm}.  Let $\alpha$ be a simple root, and let $P_\alpha \supset B$ be the corresponding parabolic subgroup.  Let $\fu_\alpha \subset \fu$ be the Lie algebra of the unipotent radical of $P_\alpha$.  The cotangent bundle $T^*(G/P_\alpha)$ can be identified with $G \times^{P_\alpha} \fu_\alpha$.  On the other hand, let $\tcN_\alpha = G \times^B \fu_\alpha$.  There are natural maps
\[
T^*(G/P_\alpha) \xleftarrow{\pi_\alpha} \tcN_\alpha \xrightarrow{i_\alpha} \tcN,
\]
where $i_\alpha$ is an inclusion of a smooth subvariety of codimension~$1$, and $\pi_\alpha$ is a smooth, proper map whose fibers are isomorphic to $P_\alpha/B \cong \mathbb{P}^1$.  

Let $\rho = \frac{1}{2} \sum_{\alpha \in \Phi^+} \alpha$, where $\Phi^+$ is the set of positive roots.  In general, $\rho$ need not lie in $\bX$, but because $G$ has simply-connected derived group, there exists a weight $\hrho \in \bX$ such that $(\beta^\vee,\hrho) = 1$ for all simple coroots $\beta^\vee$.

Let $\Psi_\alpha: \Db\Coh^{G \times \Gm}(\tcN) \to \Db\Coh^{G \times \Gm}(\tcN)$ be the functor
\[
\Psi_\alpha(\cF) = i_{\alpha*}\pi_\alpha^*\pi_{\alpha*}i_\alpha^*(\cF \otimes \cO_\tcN(\hrho-\alpha)) \otimes \cO_\tcN(-\hrho)\la 1\ra.
\]
This coincides with the functor denoted $F_\alpha\la-1\ra[1]$ or $F'_\alpha\la 1\ra[-1]$ in~\cite[\S2.3]{bez:ctm}. (Note that the statement of~\cite[Lemma~6(b)]{bez:ctm} contains a misprint.)

\begin{prop}\label{prop:braid}
\begin{enumerate}
\item The functor $\Psi_\alpha$ is self-adjoint.\label{it:braid-adjoint}
\item If $s_\alpha\lambda = \lambda$, then $\Psi_\alpha(\hDelta_\lambda) = \Psi_\alpha(\hnabla_\lambda) = 0$.\label{it:braid-kill}
\item If $s_\alpha\lambda \prec \lambda$, then\label{it:braid-short}
\[
\Psi_\alpha(\hDelta_{s_\alpha\lambda}) \cong \Psi_\alpha(\hDelta_\lambda)\la -1\ra[1]
\qquad\text{and}\qquad
\Psi_\alpha(\hnabla_{s_\alpha\lambda}) \cong \Psi_\alpha(\hnabla_\lambda)\la 1\ra[-1].
\]
\item There are natural distinguished triangles\label{it:braid-ses}
\[
\hDelta_{s_\alpha\lambda} \to \hDelta_\lambda\la 1\ra \to \Psi_\alpha(\hDelta_\lambda)[1] \to, \qquad
\Psi_\alpha(\hnabla_\lambda)[-1] \to \hnabla_\lambda\la -1\ra \to \hnabla_{s_\alpha\lambda} \to.
\]
\end{enumerate}
\end{prop}
\begin{proofsk}
For costandard objects, parts~\eqref{it:braid-adjoint}, \eqref{it:braid-short}, and~\eqref{it:braid-ses} appear in~\cite[Lemma~6 and Proposition~7]{bez:ctm}. In fact, the proofs of those statements also establish part~\eqref{it:braid-kill}.  Similar arguments apply in the case of standard objects.
\end{proofsk}

It is likely that the distinguished triangles in part~\eqref{it:braid-ses} above are actually short exact sequences in $\ExCoh^{G \times \Gm}(\tcN)$; see Section~\ref{ss:socles}.

\begin{prop}\label{prop:mu-exact}
The functor $\pi_*: \Db\Coh^{G \times \Gm}(\tcN) \to \Db\Coh^{G \times \Gm}(\cN)$ restricts to an exact functor $\pi_*: \ExCoh^{G \times \Gm}(\tcN) \to \PCoh^{G \times \Gm}(\cN)$. It satisfies
\[
\begin{aligned}
\pi_*\hDelta_\lambda &\cong \bDelta_{\dom(\lambda)}\la -\delta^*_\lambda\ra, \\
\pi_*\hnabla_\lambda &\cong \bnabla_{\dom(\lambda)}\la \delta^*_\lambda\ra,
\end{aligned}
\qquad\qquad
\pi_*\fE_\lambda \cong
\begin{cases}
\fIC_{w_0\lambda} &\text{if $\lambda \in -\bXp$,} \\
0 & \text{otherwise.}
\end{cases}
\]
\end{prop}
\begin{proofsk}
One first shows that $\pi_* \circ \Psi_\alpha = 0$ for all simple roots $\alpha$.  The formula for $\pi_*\hDelta_\lambda$ (resp.~$\pi_*\hnabla_\lambda$) is clear when $\lambda$ is dominant (resp.~antidominant), so it follows for general $\lambda$ from Proposition~\ref{prop:braid}.  The $t$-exactness of $\pi_*$ follows immediately from its behavior on standard and costandard objects.

Since $\fE_\lambda$ is the image of a nonzero map $h: \hDelta_\lambda \to \hnabla_\lambda$, $\pi_*\fE_\lambda$ is the image of $\pi_*h: \bDelta_{\dom(\lambda)}\la -\delta^*_\lambda \ra \to \bnabla_{\dom(\lambda)}\la \delta^*_\lambda\ra$.  But it follows from the definition of a graded quasiexceptional set (see Theorem~\ref{thm:pcoh-defn}) that $\Hom(\bDelta_{\dom(\lambda)}\la n \ra, \bnabla_{\dom(\lambda)}\la m\ra) = 0$ unless $n = m$.  Thus, $\pi_* h = 0$ unless $\delta^*_\lambda = 0$.  In other words, $\pi_* \fE_\lambda = 0$ if $\lambda \notin -\bXp$.

Assume now that $\lambda \in -\bXp$.  To show that $\pi_*\fE_\lambda \cong \fIC_{\dom(\lambda)}$, it suffices to show that $\pi_*h$ is nonzero.  Recall that $\hDelta_\lambda \cong \cO_\tcN(\lambda)\la \delta_\lambda\ra$.  We may assume that $h$ is the map appearing in the second distinguished triangle from Proposition~\ref{prop:mut}: $\cK_\lambda' \to \hDelta_\lambda \overset{h}{\to} \hnabla_\lambda \to$.  Since $\cK'_\lambda$ lies in $\Db\Coh^{G \times \Gm}(\tcN)_{\tl \lambda}$, our computation of $\pi_*$ on standard and costandard objects implies that
\begin{equation}\label{eqn:mu-exact}
\pi_*\cK'_\lambda \in \Db\Coh^{G \times \Gm}(\cN)_{< \dom(\lambda)}.
\end{equation}
If $\pi_*h = 0$, then $\pi_*\cK_\lambda' \cong \bDelta_{\dom(\lambda)} \oplus \bnabla_{\dom(\lambda)}[-1]$, contradicting~\eqref{eqn:mu-exact}.  Thus, $\pi_*h \ne 0$, as desired.
\end{proofsk}

\subsection{Remarks on grading choices}
\label{ss:ungraded}

Equation~\eqref{eqn:bnabla-defn} involves a choice of normalization in the grading shifts.  The choice we have made here (which is consistent with~\cite{a:pcsnc, min:mfpcs, ar:psag}) has the desirable property that it behaves well under Serre--Grothendieck duality (see Section~\ref{ss:local}).  

Proposition~\ref{prop:mut} also involves such a choice.  Our choice agrees with~\cite{ar:agsr} but differs from~\cite{bez:ctm}.  The choice we have made here has two advantages: (i)~it is compatible with the choice in~\eqref{eqn:bnabla-defn}, in the sense that the formula for $\pi_*\fE_\lambda$ involves no grading shift; and (ii)~it behaves well under Verdier duality via the derived equivalence of Section~\ref{ss:aff-grass}.  The reader should bear this in mind when comparing statements in this paper to~\cite{bez:ctm}.

One can also drop the $\Gm$-equivariance entirely and carry out all the constructions in the preceding sections in $\Db\Coh^G(\tcN)$ or $\Db\Coh^G(\cN)$.  Simple objects in $\ExCoh^G(\tcN)$ are parametrized by $\bX$ rather than by $\bX \times \Z$, and likewise for $\PCoh^G(\cN)$.  Almost all results in the paper have analogues in this setting, obtained just by omitting grading shifts $\la n\ra$ and by replacing all occurrences of $\uHom$ by $\Hom$.  In most cases, we will not treat these analogues explicitly.

However, there are a handful of exceptions.  The proof of Theorem~\ref{thm:pcoh-dereq} is different in the graded and ungraded cases.  Theorem~\ref{thm:excoh-gr} does not (yet?) have an ungraded version in positive characteristic.  Although Theorems~\ref{thm:excoh-iw} and~\ref{thm:lv} both have graded analogues, the focus in the literature and in applications is on the ungraded case, and their statements here reflect that.

\section{Structure theory I}

\subsection{Properly stratified categories}
\label{ss:prop-strat}

In this subsection, we let $\bk$ be an arbitrary field.  Let $\sA$ be a $\bk$-linear abelian category that is equipped with an automorphism $\la 1 \ra: \sA \to \sA$, called the \emph{Tate twist}.  Assume that this category has the following properties:
\begin{enumerate}
\item Every object has finite length.\label{it:fin-length}
\item For every simple object $L \in \sA$, we have $\uEnd(L) \cong \bk$.
\end{enumerate}
Let $\Omega$ be the set of isomorphism classes of simple objects up to Tate twist.  For each $\gamma \in \Omega$, choose a representative $L_\gamma$.  Assume that $\Omega$ is equipped with a partial order $\le$.  For any finite order ideal $\Gamma \subset \Omega$, let $\sA_\Gamma \subset \sA$ be the full subcategory consisting of objects all of whose composition factors are isomorphic to some $L_\gamma\la n\ra$ with $\gamma \in \Gamma$.

Such a category is of particular interest when for each $\gamma \in \Gamma$, there exist four morphisms $\bDelta_\gamma \to L_\gamma$, $L_\gamma \to \bnabla_\gamma$, $\Delta_\gamma \to L_\gamma$, $L_\gamma \to \nabla_\gamma$ such that:
\begin{enumerate}\setcounter{enumi}{2}
\item If $\gamma$ is a maximal element in an order ideal $\Gamma \subset \Omega$, then for all $\xi \in \Gamma \smallsetminus \{\gamma\}$, 
$\uHom(\bDelta_\gamma, L_\xi) = \uExt^1(\bDelta_\gamma, L_\xi) = 0$
and
$\uHom(L_\xi, \bnabla_\gamma) = \uExt^1(L_\xi, \bnabla_\gamma) = 0$.
\item The kernel of $\bDelta_\gamma \to L_\gamma$ and the cokernel of $L_\gamma \to \bnabla_\gamma$ both lie in $\sA_{< \gamma}$.
\item If $\gamma$ is a maximal element in an order ideal $\Gamma \subset \Omega$, then in $\sA_\Gamma$, $\Delta_\gamma \to L_\gamma$ is projective cover of $L_\gamma$, and $L_\gamma \to \nabla_\gamma$ is an injective envelope. Moreover, $\Delta_\gamma$ has a filtration whose subquotients are of the form $\bDelta_\gamma\la n\ra$, and $\nabla_\gamma$ has a filtration whose subquotients are of the form $\bnabla_\gamma\la n\ra$. 
\item $\uExt^2(\Delta_\gamma, \bnabla_\xi) = 0$ and $\uExt^2(\bDelta_\gamma, \nabla_\xi) = 0$ for all $\gamma, \xi \in \Omega$.\label{it:ext2}
\end{enumerate}
A category satisfying~\eqref{it:fin-length}--\eqref{it:ext2} is called a \emph{graded properly stratified category}.  If it happens that $\Delta_\gamma \cong \bDelta_\gamma$ and $\nabla_\gamma = \bnabla_\gamma$ for all $\gamma \in \Gamma$, then we instead call it a \emph{graded quasihereditary} or a \emph{highest-weight category}.  

Objects of the form $\bDelta_\gamma\la n\ra$ (resp.~$\bnabla_\gamma\la n\ra$) are called \emph{proper standard} (resp.~\emph{proper costandard}) objects.  Those of the form $\Delta_\gamma\la n\ra$ (resp.~$\nabla_\gamma\la n\ra$) are called \emph{true standard} (resp.~\emph{true costandard}) objects. In the quasihereditary case, there is no distinction between proper standard objects and true standard objects; we simply call them \emph{standard}, and likewise for \emph{costandard}.

There are obvious ungraded analogues of these notions: we omit the Tate twist, and replace all occurences of $\uHom$ and $\uExt$ above by ordinary $\Hom$ and $\Ext$.

\begin{rmk}\label{rmk:qhered}
Some sources, such as~\cite{a:pcsnc, bez:qes}, use the term \emph{quasihereditary} to refer to a category that only satisfies properties~(1)--(4) above.
\end{rmk} 

\subsection{Quasihereditarity and derived equivalences}

Let us now return to the assumptions on $\bk$ from Section~\ref{ss:notation}.  Our next goal is to see that the exotic and perverse-coherent $t$-structures fit the framework introduced above.  Because of a subtlety in the perverse-coherent case, the statements in this subsection explicitly mention the $\Gm$-equivariant and non-$\Gm$-equivariant cases separately.

\begin{thm}
The category $\ExCoh^G(\tcN)$ is quasihereditary, and the category $\ExCoh^{G \times \Gm}(\tcN)$ is graded quasihereditary.  There are equivalences of categories
\[
\Db\ExCoh^{G}(\tcN) \simto \Db\Coh^{G}(\tcN),
\qquad
\Db\ExCoh^{G \times \Gm}(\tcN) \simto \Db\Coh^{G \times \Gm}(\tcN).
\]
\end{thm}
\begin{proofsk}
The fact that these categories are (graded) quasihereditary is an immediate consequence of Theorem~\ref{thm:excoh-defn} and basic properties of exceptional sets.  For the derived equivalences, one can imitate the argument of~\cite[Corollary~3.3.2]{bgs:kdprt}.  Specifically, both $\Db\ExCoh^{G \times \Gm}(\tcN)$ and $\Db\Coh^{G \times \Gm}(\tcN)$ are generated by both standard objects and the costandard objects.  To establish the derived equivalence, it suffices to show that the natural map
\begin{equation}\label{eqn:excoh-dereq}
\Ext^k_{\ExCoh^{G \times \Gm}(\tcN)}(\hDelta_\lambda, \hnabla_\mu\la n\ra) \to \Hom_{\Db\Coh^{G \times \Gm}(\tcN)}(\hDelta_\lambda, \hnabla_\mu\la n\ra[k])
\end{equation}
is an isomorphism for all $\lambda, \mu \in \bX$, $k \ge 0$, and $n\in \Z$. When $k = 0$, this map is an isomorphism by general properties of $t$-structures.  When $k > 0$, the left-hand side vanishes by general properties of quasihereditary categories, while the right-hand side vanishes by general properties of exceptional sets. Thus,~\eqref{eqn:excoh-dereq} is always an isomorphism, as desired.  The same argument applies to $\ExCoh^G(\tcN)$.
\end{proofsk}

\begin{thm}\label{thm:pcoh-dereq}
The category $\PCoh^G(\cN)$ is properly stratified, and the category $\PCoh^{G \times \Gm}(\cN)$ is graded properly stratified. There are equivalences of categories
\[
\Db\PCoh^{G \times \Gm}(\cN) \simto \Db\Coh^{G \times \Gm}(\cN), 
\qquad
\Db\PCoh^{G}(\cN) \simto \Db\Coh^{G}(\cN).
\]
\end{thm}
\begin{proofsk}
It was shown in~\cite{bez:qes,a:pcsnc} that $\PCoh^{G \times \Gm}(\cN)$ satisfies axioms~(1)--(4) of Section~\ref{ss:prop-strat} (cf.~Remark~\ref{rmk:qhered}), with the objects of~\eqref{eqn:bnabla-defn} playing the roles of the proper standard and proper costandard objects.

The proof of axioms~(5) and~(6) is due to Minn-Thu-Aye~\cite[Theorem~4.3]{min:mfpcs}.  His argument includes a recipe for constructing the true standard and true costandard objects.  This recipe is reminiscent of Proposition~\ref{prop:mut}: specifically, by~\cite[Definition~4.2]{min:mfpcs}, 
there are canonical distinguished triangles
\[
\Delta_\lambda \to V(\lambda) \otimes \cO_\cN\la \delta^*_\lambda\ra \to \cK_\lambda \to
\qquad\text{and}\qquad
\cK'_\lambda \to H^0(\lambda) \otimes \cO_\cN\la -\delta^*_\lambda\ra \to \nabla_\lambda \to
\]
with $\cK_\lambda, \cK'_\lambda \in \Db\Coh^{G \times \Gm}(\cN)_{< \lambda}$.

The derived equivalences are more difficult here than in the exotic case, mainly because there may be nontrivial higher $\Ext$-groups between proper standard and proper costandard objects. In the $(G \times \Gm)$-equivariant case, the result is proved in~\cite{a:pcsnc}.  The proof makes use of the $\Gm$-action in an essential way; it cannot simply be copied in the ungraded case.  However, by~\cite[Lemma~A.7.1]{bei:dcps}, the following diagram commutes up to natural isomorphism:
\[
\xymatrix{
\Db\PCoh^{G \times \Gm}(\cN) \ar[r]^{\sim} \ar[d]_U &
  \Db\Coh^{G \times \Gm}(\cN) \ar[d]^U \\
\Db\PCoh^{G}(\cN) \ar[r] &
  \Db\Coh^{G}(\cN) }
\]
Here, both vertical arrows are the functors that forget the $\Gm$-equivariance.  It is not difficult to check that these vertical arrows are \emph{degrading functors} as in~\cite[\S4.3]{bgs:kdprt}.  Thus, for any $\lambda, \mu \in \bX$, we have a commutative diagram
\[
\xymatrix@C=15pt{
\bigoplus_{n \in \Z} \Ext^k_{\PCoh^{G \times \Gm}(\cN)}(\bDelta_\lambda, \bnabla_\mu\la n\ra) \ar[r]^-{\sim} \ar[d]_{\wr} &
\bigoplus_{k \in \Z} \Hom_{\Db\Coh^{G \times \Gm}(\cN)}(\bDelta_\lambda, \bnabla_\mu\la k\ra[n]) \ar[d]^{\wr} \\
\Ext^k_{\PCoh^G(\cN)}(U(\bDelta_\lambda), U(\bnabla_\mu)) \ar[r] &
\Hom_{\Db\Coh^G(\cN)}(U(\bDelta_\lambda), U(\bnabla_\mu)[k])}
\]
Since the top arrow and both vertical arrows are isomorphisms, the bottom arrow must be as well.  That map is analogous to~\eqref{eqn:excoh-dereq}, and, as in the proof of the preceding theorem, it implies that $\Db\PCoh^G(\cN) \cong \Db\Coh^G(\cN)$.
\end{proofsk}

\subsection{Costandard and tilting objects}
\label{ss:costandard}

The abstract categorical framework of Section~\ref{ss:prop-strat} places standard and costandard objects on an equal footing, but in practice, the following result makes costandard objects considerably easier to work with explicitly.

\begin{thm}\label{thm:costd-coh}
\begin{enumerate}
\item For all $\lambda \in \bX$, $\hnabla_\lambda$ is a coherent sheaf.
\item For all $\lambda \in \bXp$, $\bnabla_\lambda$ is a coherent sheaf.
\end{enumerate}
\end{thm}
In contrast, even for $\SL_2$, many standard objects are complexes with cohomology in more than one degree.
\begin{proofsk}
The first assertion is proved in~\cite{ar:agsr}, using Theorem~\ref{thm:excoh-gr} below to translate it into a question about the dual affine Grassmannian.  The second assertion is due to Kumar--Lauritzen--Thomsen~\cite{klt:fscb}, although in many cases it goes back to much older work of Andersen--Jantzen~\cite{aj:cirag}.
\end{proofsk}

Recall that a \emph{tilting object} in a quasihereditary category is one that has both a standard filtration and a costandard filtration.  The isomorphism classes of indecomposable tilting objects are in bijection with the isomorphism classes of simple (or standard, or costandard) objects.  Let
\[
\hfT_\lambda \in \ExCoh^{G \times \Gm}(\tcN)
\]
denote the indecomposable tilting object corresponding to $\fE_\lambda$.

In a properly stratified category that is not quasihereditary, there are two distinct versions of this notion: an object is called \emph{tilting} if it has a true standard filtration and a proper costandard filtration, and \emph{cotilting} if it has a proper standard filtration and a true costandard filtration.  These notions need not coincide in general.  See~\cite[\S2.2]{ar:psag} for general background on (co)tilting objects in this setting.

\begin{prop}[\cite{min:mfpcs}]\label{prop:pcoh-tilt}
In $\PCoh^{G \times \Gm}(\cN)$, the indecomposable tilting and cotilting objects coincide, and they are all of the form $T(\lambda) \otimes \cO_\cN\la n\ra$.
\end{prop}

\begin{prop}\label{prop:tilt-dom}
In $\ExCoh^{G \times \Gm}(\tcN)$, every tilting object is a coherent sheaf.  For $\lambda \in \bXp$, we have $\hfT_\lambda \cong T(\lambda) \otimes \cO_\tcN$.
\end{prop}

\begin{proofsk}
This can be deduced by adjunction from Proposition~\ref{prop:pcoh-tilt} using the fact that $\pi^*\cO_\cN \cong \pi^!\cO_\cN \cong \cO_\tcN$ and the criterion from~\cite[Lemma~4]{bez:ctm}.
\end{proofsk}

At the moment, there is no comparable statement describing $\hfT_\lambda$ for nondominant $\lambda$.  A better understanding of tilting exotic sheaves is highly desirable; in particular, it would shed light on the  question below.  An affirmative answer would have significant consequences for the geometry of affine Grassmannians and for modular representation theory. In Section~\ref{sect:sl2}, we will answer this question for $\SL_2$.

\begin{ques}[Positivity for tilting exotic sheaves]\label{ques:positivity}
Is it true that $\uHom(\hfT_\lambda, \hfT_\mu)$ is concentrated in nonnegative degrees for all $\lambda, \mu \in \bX$?
\end{ques}

\section{Applications}
\label{sect:apps}

Many of the applications of exotic and perverse-coherent sheaves rely on the fact that these $t$-structures can be constructed in several rather different ways.  
\begin{description}
\item[Quasiexceptional sets] This refers to the construction that was carried out in Section~\ref{sect:defn}. (For an explanation of this term and additional context, see~\cite[\S2.2]{bez:qes} and~\cite[\S2.1]{bez:ctm}.) 
\item[Whittaker sheaves] In this approach, we transport the natural $t$-structure across a derived equivalence relating our category of coherent sheaves to a suitable category of Iwahori--Whittaker perverse sheaves on the affine flag variety for the Langlands dual group $\Gv$.
\item[Affine Grassmannian] This approach is Koszul dual to the preceding one, and involves Iwahori-monodromic perverse sheaves on the affine Grassmannian for $\Gv$.
\item[Local cohomology] In the perverse-coherent case, there is an algebro-geo\-metric construction in terms of local cohomology that superficially resembles the definition of ordinary (constructible) perverse sheaves. Indeed, this description is the reason for the name ``perverse-coherent.''
\item[Braid positivity] In the exotic case, the $t$-structure is uniquely determined by certain exactness properties, the most important of which involves the affine braid group action of~\cite{br:abga}.  
\end{description}
In this section, we briefly review these various approaches and discuss some of their applications. Some of these are (for now?) available only in characteristic zero.

\subsection{Whittaker sheaves and quantum group cohomology}
\label{ss:whittaker}

Let $p$ be a prime number, and let $\bK = \bFp((t))$ and $\bO = \bFp[[t]]$.  Let $\Gv$ be the Langlands dual group to $G$ over $\bFp$, and let $\Bv^-, \Bv \subset \Gv$ be opposite Borel subgroups corresponding to negative and positive roots, respectively.  (Note that our convention for $\Gv$ differs from that for $G$, where $B \subset G$ denotes a negative Borel subgroup.)  These groups determine a pair $I^- = e^{-1}(\Bv^-)$, $I = e^{-1}(\Bv)$ of oppposite Iwahori subgroups, where $e: \Gv_\bO \to \Gv$ is the map induced by $t \mapsto 0$. Recall that the \emph{affine flag variety} is the space $\Fl = \Gv_\bK/I$.

Let $\Uv^- \subset \Bv^-$ be the unipotent radical, and for each simple root $\alpha$, let $\Uv^-_\alpha \subset \Uv^-$ be the root subgroup corresponding to $-\alpha$.  The quotient $\Uv^-/[\Uv^-,\Uv^-]$ can be identified with the product $\prod_{\alpha} \Uv^-_\alpha$.  For each $\alpha$, fix an isomorphism $\psi_\alpha: \Uv^-_\alpha \cong \Ga$.
Let $I^-_u = e^{-1}(\Uv^-)$ be the pro-unipotent radical of $I^-$, and let $\psi: I^-_u \to \Ga$ be the composition
\[
I^-_u \xrightarrow{e} \Uv^- \to \Uv^-/[\Uv^-,\Uv^-] \cong \prod_\alpha U^-_\alpha \xrightarrow{\prod \psi_\alpha} \prod_\alpha \Ga \xrightarrow{\sum} \Ga.
\]
Finally, let $\cX = \psi^*\mathrm{AS}$, where $\mathrm{AS}$ denotes an Artin--Schreier local system on $\Ga$.

Let $\ell$ be a prime number different from $p$.  The \emph{Iwahori--Whittaker derived category} of $\Fl$, denoted $\Db_\IW(\Fl,\Qlb)$, is defined to be the $(I^-_u,\cX)$-equivariant derived category of $\Qlb$-sheaves on $\Fl$.  (In many sources, this is simply called the $(I^-_u,\psi)$-equivariant derived category.  For background on this kind of equivariant derived category, see, for instance,~\cite[Appendix~A]{ar:mpsfv1}.)  We also have the abelian category $\Perv_\IW(\Fl,\Qlb)$ of \emph{Iwahori--Whittaker perverse sheaves} on $\Fl$.  

\begin{thm}[\cite{ab:psaf,bez:ctm}]\label{thm:excoh-iw}
Assume that $\bk = \Qlb$.  There is an equivalence of triangulated categories
\[
\Db\Coh^G(\tcN) \cong \Db_\IW(\Fl,\Qlb).
\]
This equivalence is $t$-exact for the exotic $t$-structure on the left-hand side and the perverse $t$-structure on the right-hand side.  In particular, there is an equivalence of abelian categories
\[
\ExCoh^G(\tcN) \cong \Perv_\IW(\Fl,\Qlb).
\]
\end{thm}

There is an equivalence of categories $\Db\Perv_\IW(\Fl,\Qlb) \simto \Db_\IW(\Fl,\Qlb)$ (see \cite[Lemma~1]{ab:psaf}), so Theorem~\ref{thm:excoh-iw} can be restated in a way that matches the exotic $t$-structure with the natural $t$-structure on $\Db\Perv_\IW(\Fl,\Qlb)$.

This equivalence plays a key role in Bezrukavnikov's computation of the cohomology of tilting modules for quantum groups at a root of unity~\cite{bez:ctm}.  Specifically, after relating $\Db\Coh^{G \times \Gm}(\tcN)$ to the derived category of the principal block of the quantum group, the desired facts about quantum group cohomology are reduced to the following statement about exotic sheaves, called the positivity lemma~\cite[Lemma~9]{bez:ctm}:
\begin{equation}\label{eqn:positivity-lemma}
\Ext^i(\hDelta_\lambda\la n\ra, \fE_\mu) = \Ext^i(\fE_\mu, \hnabla_\lambda\la -n\ra) = 0\qquad\text{if $i > n$.}
\end{equation}
To prove the positivity lemma, one uses Theorem~\ref{thm:excoh-iw} to translate it into a question about Weil perverse sheaves on $\Fl$.  The latter can be answered using the powerful and well-known machinery of~\cite{bbd}.

\bigskip

We will not discuss the proof of Theorem~\ref{thm:excoh-iw}, but as a plausibility check, let us review the parametrization of simple objects in $\Perv_\IW(\Fl,\Qlb)$.  Iwahori--Whittaker perverse sheaves are necessarily constructible along the $I^-$-orbits on $\Fl$, which, like the $I$-orbits, are naturally parametrized by the extended affine Weyl group $W_\ext$ for $G$.  However, not every $I^-$-orbit supports an $\cX$-equivariant local system: according to~\cite[Lemma~2]{ab:psaf}, those that do correspond to the set $^f W_\ext \subset W_\ext$ of minimal-length coset representatives for $W \backslash W_\ext$.  Thus, simple objects in $\Perv_\IW(\Fl,\Qlb)$ are parametrized by $^f W_\ext$, which is naturally in bijection with $\bX$.

For $w \in {}^f W_\ext$, let $L_w \in \Perv_\IW(\Fl,\Qlb)$ denote the corresponding simple object.  Let $^f W^f_\ext$ be the set of minimal-length representatives for the double cosets $W \backslash W_\ext / W$, and form the Serre quotient
\[
\Perv^f_\IW(\Fl,\Qlb) = \Perv_\IW(\Fl,\Qlb)\bigg/\left(
\begin{array}{c}
\text{the Serre subcategory generated} \\
\text{by the $L_w$ with $w \notin {}^fW^f_\ext$}
\end{array}
\right).
\]

\begin{thm}[\cite{bez:psaf}]\label{thm:pcoh-iw}
Assume that $\bk = \Qlb$.  There is an equivalence of triangulated categories
\[
\Db\Coh^G(\cN) \cong \Db\Perv^f_\IW(\Fl,\Qlb).
\]
This equivalence is $t$-exact for the perverse-coherent $t$-structure on the left-hand side and the natural $t$-structure on the right-hand side.  In particular, there is an equivalence of abelian categories
\[
\PCoh^G(\cN) \cong \Perv^f_\IW(\Fl,\Qlb).
\]
\end{thm}

It seems likely that analogous statements to the theorems in this subsection hold when $\bk$ has positive characteristic.

\subsection{The affine Grassmannian and the Mirkovi\'c--Vilonen conjecture}
\label{ss:aff-grass}

Recall that the \emph{affine Grassmannian} is the space $\Gr = \Gv_\bK/\Gv_\bO$.  Here, we may either define $\bK$ and $\bO$ as in the previous subsection, and work with \'etale sheaves on $\Gr$, or we may instead put $\bK = \C((t))$ and $\bO = \C[[t]]$ and equip $\Gr$ with the classical topology.  (For a discussion of how to compare the two settings, see, e.g.,~\cite[Remark~7.1.4(2)]{rsw:mkd}.)  In this subsection, we will work with a certain category of ``mixed'' $I$-monodromic perverse $\bk$-sheaves on $\Gr$, denoted $\Perv^\mix_{(I)}(\Gr,\bk)$.
If $\bk$ has characteristic zero, this category  should be defined following the pattern of~\cite[Theorem~4.4.4]{bgs:kdprt} or~\cite[\S6.4]{ar:kdsf}: $\Perv^\mix_{(I)}(\Gr,\Qlb)$ is \emph{not} the category of all mixed perverse sheaves in the sense of~\cite{bbd}, but rather the full subcategory in which we allow only Tate local systems and require the associated graded of the weight filtration to be semisimple.  For $\bk$ of positive characteristic, this category is defined in~\cite{ar:mpsfv2} in terms of the homological algebra of parity sheaves.

In both cases, the additive category $\Parity_{(I)}(\Gr,\bk)$ of Iwahori-constructible parity sheaves can be identified with a full subcategory of $\Db\Perv^\mix_{(I)}(\Gr,\bk)$.  In the case where $\bk = \Qlb$, $\Parity_{(I)}(\Gr,\bk)$ is identified with the category of pure semisimple complexes of weight~$0$.  

\begin{thm}[\cite{abg:qglg} for $\bk = \Qlb$; \cite{ar:agsr,mr:etsps} in general]
\label{thm:excoh-gr}
There is an equivalence of triangulated categories
\[
P: \Db\Coh^{G \times \Gm}(\tcN) \simto \Db\Perv^\mix_{(I)}(\Gr,\bk)
\]
such that $P(\cF\la n\ra) \cong P(\cF)(\frac{n}{2})[n]$. This equivalence is not $t$-exact, but it does induce an equivalence of additive categories
\[
\Tilt(\ExCoh^{G \times \Gm}(\tcN)) \simto \Parity_{(I)}(\Gr,\bk).
\]
\end{thm}

Note that the exotic $t$-structure can be recovered from the class of tilting objects in its heart.  When $\bk = \Qlb$, there is also an ``unmixed'' version of this theorem~\cite{abg:qglg}.  In positive characteristic, a putative unmixed statement is equivalent to a ``modular formality'' property for $\Gr$ that is currently still open.

The next theorem is a similar statement for the perverse-coherent $t$-structure.  This result does not, however, extend to an equivalence involving the full derived category $\Db\Coh^{G \times \Gm}(\cN)$.

\begin{thm}[\cite{ar:psag}]\label{thm:pcoh-gr}
There is an equivalence of additive categories
\[
\Tilt(\PCoh^{G \times \Gm}(\cN)) \simto \Parity_{(\Gv_\bO)}(\Gr,\bk).
\]
\end{thm}

An important consequence of the preceding theorem is the following result, known as the Mirkovi\'c--Vilonen conjecture (see~\cite[Conjecture~6.3]{mv:psag} or~\cite[Conjecture~13.3]{mv:gld}).  In bad characteristic, the conjecture is false~\cite{jut:mrrg}.

\begin{thm}[\cite{ar:psag}]
Under the geometric Satake equivalence, the stalks (resp.\ costalks) of the perverse sheaf on $\Gr$ corresponding to a Weyl module (resp.\ dual Weyl module) of $G$ vanish in odd degrees.
\end{thm}
\begin{proofsk}
The statement we wish to prove can be rewritten as a statement about the vanishing of certain $\Ext$-vanishing groups in the derived category of constructible complexes of $\bk$-sheaves on $\Gr$.  Theorem~\ref{thm:pcoh-gr} lets us translate that question into one about $\Hom$-groups in the abelian category $\PCoh^{G \times \Gm}(\cN)$ instead.  The latter question turns out to be quite easy; it is an exercise using basic properties of properly stratified categories.
\end{proofsk}

\subsection{Local cohomology and the Lusztig--Vogan bijection}
\label{ss:local}

The following theorem describes $\PCoh^{G \times \Gm}(\cN)$ in terms of cohomology-vanishing conditions on a complex $\cF$ and on its Serre--Grothendieck dual $\SGD(\cF)$, given by $\SGD(\cF) = \cRHom(\cF,\cO_\cN)$.  These conditions closely resemble the definition of ordinary (constructible) perverse sheaves; indeed, this theorem is the justification for the term ``perverse-coherent.''

\begin{thm}[\cite{bez:qes}; see also~\cite{a:pcsnc}]\label{thm:pcoh-loc}
Let $\cF \in \Db\Coh^{G \times \Gm}(\cN)$.  The following conditions are equivalent:
\begin{enumerate}
\item $\cF$ lies in $\PCoh^{G \times \Gm}(\cN)$.
\item We have $\dim \supp \cH^i(\cF) \le \dim \cN - 2i$ and $\dim \supp \cH^i(\SGD\cF) \le \dim \cN - 2i$ for all $i \in \Z$.\label{it:dim-supp}
\item Whenever $x \in \cN$ is a generic point of a $G$-orbit, we have $H^i(\cF_x) = 0$ if $i > \frac{1}{2}\codim \bar x$, and $H^i_x(\cF) = 0$ if $i < \frac{1}{2}\codim \bar x$.\label{it:loc-coh}
\end{enumerate}
\end{thm}
(In the last assertion, $\cF_x$ is just the stalk of $\cF$ at $x$, while $H^i_x({-})$ is cohomology with support at $x$.)
\begin{proofsk}
In~\cite{bez:qes}, condition~\eqref{it:loc-coh} was taken as the definition of the category $\PCoh^{G \times \Gm}(\cN)$, following~\cite{bez:pc, ab:pcs}, while the $t$-structure of Theorem~\ref{thm:pcoh-defn} is considered separately and initially given no name.  According to~\cite[Corollary~3]{bez:qes}, the two $t$-structures coincide; the proof consists of showing that the $A_\lambda$ satisfy condition~\eqref{it:loc-coh}.

Condition~\eqref{it:loc-coh} can be used to define ``perverse-coherent'' $t$-structures on varieties or stacks in considerable generality, not just on the nilpotent cone of a reductive group.  This theory is developed in~\cite{bez:pc,ab:pcs}.  The equivalence of conditions~\eqref{it:dim-supp} and~\eqref{it:loc-coh} holds in this general framework; see~\cite[Lemma~2.18]{ab:pcs}.
\end{proofsk}

As with ordinary perverse sheaves, there is a special class of perverse-coherent sheaves satisfying stronger dimension bounds. Let $C \subset \cN$ be a nilpotent orbit, and let $\cE$ be a $(G \times \Gm)$-equivariant vector bundle on $C$.  There is an object
\[
\fIC(C,\cE) \in \PCoh^{G \times \Gm}(\cN),
\]
called a \emph{coherent intersection cohomology complex}, that is uniquely characterized by the following two conditions:
\begin{enumerate}
\item $\fIC(C,\cE)$ is supported on $\overline{C}$, and $\fIC(C,\cE)|_C \cong \cE[-\frac{1}{2}\codim C]$.
\item We have $\dim \supp \cH^i(\cF) < \dim \cN - 2i$ and $\dim \supp \cH^i(\SGD\cF) < \dim \cN - 2i$ for all $i > \frac{1}{2}\codim C$.
\end{enumerate}
Moreover, when $\cE$ is an irreducible vector bundle, $\fIC(C,\cE)$ is a simple object of $\PCoh^{G \times \Gm}(\cN)$, and every simple object arises in this way.

Theorem~\ref{thm:pcoh-loc} has an obvious $G$-equivariant analogue (omitting the $\Gm$-equi\-var\-i\-ance), as does the notion of coherent intersection cohomology complexes.  The latter yields a bijection
\begin{equation}\label{eqn:pcoh-ic}
\left\{
\begin{array}{@{}c@{}}
\text{simple objects} \\
\text{in $\PCoh^G(\cN)$}
\end{array}
\right\} \overset{\sim}{\longleftrightarrow}
\left\{(C,\cE) \mathbin{\Big|}
\begin{array}{@{}c@{}}
\text{$C$ a $G$-orbit, $\cE$ an irreducible} \\
\text{$G$-equivariant vector bundle on $C$}
\end{array}
\right\}
\end{equation}
that looks very different from the parametrization of simple objects in~\S\ref{ss:pcoh}.  Comparing the two yields the following result of Bezrukavnikov.

\begin{thm}[\cite{bez:qes}]\label{thm:lv}
There is a canonical bijection
\[
\bXp \overset{\sim}{\longleftrightarrow} \{(C,\cE)\}.
\]
\end{thm}

The existence of such a bijection was independently conjectured by Lusztig~\cite{lus:cawg4} and Vogan.  For $G = \mathrm{GL}_n$, the Lusztig--Vogan bijection was established earlier~\cite{a:phd} (see also~\cite{a:ekt}) by an argument that provided an explicit combinatorial description of the bijection.  

In general, it is rather difficult to carry out computations with coherent $\fIC$'s, and the problem of computing the Lusztig--Vogan bijection explicitly remains open in most cases.  The extreme cases corresponding to the regular and zero nilpotent orbits are discussed below, following~\cite[Proposition~2.8]{a:ekt}.  Let
\[
\cN_\reg \subset \cN
\qquad\text{and}\qquad
C_0 \subset \cN
\]
denote the regular and zero nilpotent orbits, respectively.

\begin{prop}\label{prop:lv-regular}
The bijection of Theorem~\ref{thm:lv} restricts to a bijection
\[
\bXmin \overset{\sim}{\longleftrightarrow} \{(\cN_\reg,\cE)\}
\]
\end{prop}
\begin{proofsk}
Since the elements of $\bXmin$ are precisely the minimal elements of $\bXp$ with respect to $\preceq$ (or $\le$), the proper costandard objects $\{ \bnabla_\lambda \mid \lambda \in \bXmin \}$ are simple.  Every $A_\lambda$ has nonzero restriction to $\cN_\reg$ (since $\pi$ is an isomorphism over $\cN_\reg$), so for $\lambda \in \bXmin$, the simple object $\fIC_\lambda = \bnabla_\lambda$ must coincide with some $\fIC(\cN_\reg,\cE)$.  Thus, the bijection of Theorem~\ref{thm:lv} at least restricts to an injective map $\bXmin \hookrightarrow \{(\cN_\reg,\cE)\}$.  The fact that it is also surjective can be deduced from the well-known relationship between minuscule weights and representations of the center of $G$.
\end{proofsk}

\begin{prop}\label{prop:lv-zero}
The bijection of Theorem~\ref{thm:lv} restricts to a bijection
\[
\bXp + 2\rho \overset{\sim}{\longleftrightarrow} \{(C_0,\cE)\}.
\]
\end{prop}

Here, $2\rho = \sum_{\alpha \in \Phi^+} \alpha$, as in Section~\ref{ss:mu-exact}.  The proof of this will be briefly discussed at the end of Section~\ref{ss:charform}.

\subsection{Affine braid group action and modular representation theory}

In~\cite{bm:rssla}, Bezrukavnikov and Mirkovi\'c proved a collection of conjectures of Lusztig \cite{lus:bekt2} involving the equivariant $K$-theory of Springer fibers and the representation theory of semisimple Lie algebras in positive characteristic.  In this work, which builds on the localization theory developed in~\cite{bmr:lmsla, bmr:slif}, a key ingredient is the \emph{noncommutative Springer resolution}, a certain $\bk[\cN]$-algebra $A^0$ equipped with a $(G \times \Gm)$-action, along with a derived equivalence
\begin{equation}\label{eqn:noncomm}
\Db(A^0\lmod^{G \times \Gm}) \cong \Db\Coh^{G \times \Gm}(\tcN).
\end{equation}
Here, we will just discuss one small aspect of the argument.  At a late stage in~\cite{bm:rssla}, one learns that Lusztig's conjectures follow from a certain positivity statement about graded $A^0$-modules, and that, moreover, it is enough to prove that positivity statement in characteristic~$0$.  From then on, the proof follows the pattern we saw with~\eqref{eqn:positivity-lemma}: by composing~\eqref{eqn:noncomm} with Theorem~\ref{thm:excoh-iw}, one can translate the desired positivity statement into a statement about Weil perverse sheaves on the affine flag variety $\Fl$, and then use the machinery of weights from~\cite{bbd}.

To carry out the ``translation'' step, one needs an appropriate description of the $t$-structure on $\Db(A^0\lmod^{G \times \Gm})$ corresponding to the perverse $t$-structure on $\Db\Perv_\IW(\Fl,\Qlb)$, or, equivalently, to the exotic $t$-structure\footnote{A caveat about terminology: most of~\cite{bm:rssla} is concerned with \emph{nonequivariant} coherent sheaves or $A^0$-modules.  In that paper, the term \emph{exotic $t$-structure} refers to a certain $t$-structure in the nonequivariant setting, and not to the $t$-structure of Theorem~\ref{thm:excoh-defn}.  In~\cite{bm:rssla}, the latter $t$-structure is instead called \emph{perversely exotic}.} on $\Db\Coh^{G \times \Gm}(\tcN)$.  Unfortunately, the exceptional set construction of Section~\ref{sect:defn} is ill-suited to this purpose.

The theorem below gives a new characterization of $\ExCoh^{G \times \Gm}(\tcN)$ that does adapt well to the setting of $A^0$-modules.  Its key feature is the prominent role it gives to the \emph{affine braid group action} on $\Db\Coh^{G \times \Gm}(\tcN)$ that is constructed in~\cite{br:abga}.  Specifically, it involves the following notion: a $t$-structure on $\Db\Coh^{G \times \Gm}(\tcN)$ is said to be \emph{braid-positive} if, in the aforementioned affine braid group action, the action of positive words in the braid group is right $t$-exact.  (The definition of the ring $A^0$ also involves braid positivity, and the $t$-structure on $\Db\Coh^{G \times \Gm}(\tcN)$ corresponding to the natural $t$-structure on $\Db(A^0\lmod^{G \times \Gm})$ is braid-positive.)

Before stating the theorem, we need some additional notation and terminology.  Given a closed $G$-stable subset $Z \subset \cN$, let $\Db_Z \Coh^{G \times \Gm}(\tcN) \subset \Db\Coh^{G \times \Gm}(\tcN)$ be the full subcategory consisting of objects supported set-theo\-ret\-i\-cally on $\pi^{-1}(Z)$.  For a nilpotent orbit $C \subset \cN$, let $\Db_C\Coh^{G \times \Gm}(\tcN)$ be the quotient category
\[
\Db_{\overline{C}}\Coh^{G \times \Gm}(\tcN)/\Db_{\overline{C} \smallsetminus C} \Coh^{G \times \Gm}(\tcN).
\]
A $t$-structure on $\Db\Coh^{G \times \Gm}(\tcN)$ is said to be \emph{compatible with the support filtration} if for every nilpotent orbit $C$, there are induced $t$-structures on $\Db_{\overline{C}}\Coh^{G \times \Gm}(\tcN)$ and $\Db_C\Coh^{G \times \Gm}(\tcN)$ such that the inclusion and quotient functors, respectively, are $t$-exact.

\begin{thm}[{\cite[\S6.2.2]{bm:rssla}}]\label{thm:noncomm}
Assume that the characteristic of $\bk$ is zero or larger than the Coxeter number of $G$. 
The exotic $t$-structure on $\Db\Coh^{G \times \Gm}(\tcN)$  is the unique $t$-structure with all three of the following properties:
\begin{enumerate}
\item It is braid-positive.
\item It is compatible with the support filtration.
\item The functor $\pi_*$ is $t$-exact with respect to this $t$-structure and the perverse-coherent $t$-structure on $\Db\Coh^{G \times \Gm}(\cN)$.
\end{enumerate}
\end{thm}

See~\cite{dod:ecssb, mr:etsps} for comprehensive accounts of the role of the affine braid group action in the study of the exotic $t$-structure.

\section{Structure theory II}

\subsection{Minuscule objects}

As noted above, exotic sheaves do not, in general, admit a ``local'' description like that of perverse-coherent sheaves in Section~\ref{ss:local}.  Nevertheless, when we look at the set of regular elements
\[
\tcN_\reg := \pi^{-1}(\cN_\reg),
\]
the compatibility with the support filtration from Theorem~\ref{thm:noncomm} lets us identify a handful of simple exotic sheaves.  Recall that $\pi$ is an isomorphism over $\cN_\reg$, so $\tcN_\reg$, like $\cN_\reg$, is a single $G$-orbit.

\begin{lem}\label{lem:minu-exotic}
\begin{enumerate}
\item If $\lambda \in -\bXmin$, then $\fE_\lambda \cong \cO_\tcN(\lambda)\la \delta_\lambda\ra$.
\item If $\lambda \notin -\bXmin$, then $\fE_\lambda|_{\tcN_\reg} = 0$.
\end{enumerate}
\end{lem}

This lemma says that we can detect antiminuscule composition factors in an exotic sheaf by restricting to $\tcN_\reg$.

\begin{proof}
The first assertion follows from the observation that elements of $-\bXmin$ are minimal with respect to $\le$, so the standard objects $\hDelta_\lambda = \cO_\tcN(\lambda)\la \delta_\lambda\ra$ are simple.

Now take an arbitrary $\lambda \in \bX$, and suppose that $\fE_\lambda|_{\tcN_\reg} \ne 0$.  Since $\pi$ is an isomorphism over $\cN_\reg$, $\pi_*\fE_\lambda$ has nonzero restriction to $\cN_\reg$.  A fortiori, $\pi_*\fE_\lambda$ is nonzero.  By Proposition~\ref{prop:mu-exact}, $\lambda$ must be antidominant, and $\pi_*\fE_\lambda \cong \fIC_{w_0\lambda}$.  Then Proposition~\ref{prop:lv-regular} tells us that $w_0\lambda \in \bXmin$, so $\lambda \in -\bXmin$, as desired.
\end{proof}

\begin{lem}\label{lem:line-tcn-reg}
Let $\lambda, \mu \in \bX$, and suppose that $\lambda - \mu$ lies in the root lattice.  Then $\cO_\tcN(\lambda)\la (2\rho^\vee,\mu-\lambda) \ra|_{\tcN_\reg} \cong \cO_\tcN(\mu)|_{\tcN_\reg}$.
\end{lem}
\begin{proof}
Since $\lambda - \mu$ is a linear combination of simple roots, it suffices to prove the lemma in the special case where $\mu = 0$ and $\lambda$ is a simple root, say $\alpha$.  In this case, $(2\rho^\vee,-\alpha) = -2$.

Let $\tcN_\alpha$ be as in Section~\ref{ss:mu-exact}.  As explained in~\cite[Lemma~6]{bez:ctm} or~\cite[Lemma~5.3]{a:pcsnc}, there is a short exact sequence of coherent sheaves 
\begin{equation}\label{eqn:nalpha-pres}
0 \to \cO_\tcN(\alpha)\la -2\ra \to \cO_\tcN \to i_{\alpha*} \cO_{\tcN_\alpha} \to 0.
\end{equation}
Recall that $\tcN_\alpha$ does not meet $\tcN_\reg$---indeed, its image under $\pi$ is the closure of the subregular nilpotent orbit.  So when we restrict to $\tcN_\reg$, that short exact sequence gives us the desired isomorphism $\cO_\tcN(\alpha)\la -2\ra|_{\tcN_\reg} \simto \cO_\tcN|_{\tcN_\reg}$.
\end{proof}

\begin{prop}\label{prop:pcoh-torsion}
For all $\lambda \in \bXp$, $\bnabla_\lambda$ is a torsion-free coherent sheaf on~$\cN$.
\end{prop}
\begin{proof}
Let $Y = \cN \smallsetminus \cN_\reg$, and let $s: Y \hookrightarrow \cN$ be the inclusion map.  Since $\cN_\reg$ is the unique open $G$-orbit, any coherent sheaf with torsion must have a subsheaf supported on $Y$.  If $\bnabla_\lambda$ had such a subsheaf, then $\cH^0(s^!\bnabla_\lambda)$ would be nonzero.  But this contradicts the local description of $\PCoh^{G \times \Gm}(\cN)$ from Section~\ref{ss:local}.
\end{proof}

\begin{prop}\label{prop:aj-regular}
Let $\mu\in \bXmin$.  For any $\lambda \in \bX$, we have
\[
[ A_\lambda : \fIC_\mu\la n\ra] =
\begin{cases}
1 & \text{if $\mu = \smp(\lambda)$ and $n = (2\rho^\vee, \lambda)$,} \\
0 & \text{otherwise.}
\end{cases}
\]
\end{prop}
\begin{proof}
Proposition~\ref{prop:lv-regular} implies that we can determine the multiplicities of minuscule objects in any perverse-coherent sheaf by considering its restriction to $\cN_\reg$.  By Lemmas~\ref{lem:minu-exotic} and~\ref{lem:line-tcn-reg}, we have
\begin{multline*}
A_\lambda|_{\cN_\reg} \cong \pi_*\cO_\tcN(\lambda)|_{\cN_\reg} 
\cong \pi_*\cO_\tcN(\smm(\lambda))\la (2\rho^\vee, \lambda - \smm(\lambda))\ra|_{\cN_\reg} \\
\cong \pi_*\fE_{\smm(\lambda)}\la (2\rho^\vee, \lambda - \smm(\lambda)) - \delta_{\smm(\lambda)} \ra|_{\cN_\reg} \\
\cong \fIC_{\smp(\lambda)}\la (2\rho^\vee, \lambda - \smm(\lambda)) - \delta^*_{\smp(\lambda)} \ra|_{\cN_\reg}.
\end{multline*}
Consider the special case where $\lambda \in \bXmin$.  In other words, $\lambda = \smp(\lambda)$, and $A_{\smp(\lambda)} \cong \bnabla_{\smp(\lambda)}\la \delta^*_{\smp(\lambda)}\ra \cong \fIC_{\smp(\lambda)} \la \delta^*_{\smp(\lambda)}\ra$.  Comparing with the formula above, we see that
\[
(2\rho^\vee, \smp(\lambda) - \smm(\lambda)) - \delta^*_{\smp(\lambda)} = \delta^*_{\smp(\lambda)}
\]
for any $\lambda \in \bX$.  Since $(2\rho^\vee, \smp(\lambda)) = -(2\rho^\vee, \smm(\lambda))$, we deduce that
\begin{equation}\label{eqn:delta-min}
(2\rho^\vee, \smp(\lambda)) = (2\rho^\vee, -\smm(\lambda)) = \delta_{\smm(\lambda)} = \delta^*_{\smp(\lambda)}.
\end{equation}
The result follows.
\end{proof}

\begin{cor}\label{cor:aj-regular}
\begin{enumerate}
\item Let $\lambda \in \bXp$ and $\mu \in \bXmin$.  We have
\[
[ \bnabla_\lambda : \fIC_\mu\la n\ra] =
\begin{cases}
1 & \text{if $\mu = \smp(\lambda)$ and $n = (2\rho^\vee, \lambda) - \delta^*_\lambda$,} \\
0 & \text{otherwise.}
\end{cases}
\]
\item Let $\lambda \in \bX$ and $\mu \in -\bXmin$.  We have
\[
[ \hnabla_\lambda : \fE_\mu\la n\ra ] =
\begin{cases}
1 & \text{if $\mu = \smm(\lambda)$ and $n = (2\rho^\vee, \dom(\lambda)) - \delta_\lambda$,} \\
0 & \text{otherwise.}
\end{cases}
\]
\end{enumerate}
\end{cor}
\begin{proof}
The first part is just a restatement of Proposition~\ref{prop:aj-regular}.  Next, Proposition~\ref{prop:mu-exact} implies that
$
[ \hnabla_\lambda : \fE_\mu\la n\ra ] = [ \bnabla_{\dom(\lambda)}\la \delta^*_\lambda\ra : \fIC_{w_0\mu}\la n\ra ]
$.  Using the observation that $\delta^*_\lambda - \delta^*_{\dom(\lambda)} = - \delta_\lambda$, the second part follows from the first.
\end{proof}

\subsection{Character formulas}
\label{ss:charform}

In this subsection, we work with the group ring $\Z[\bX]$ and its extension $\Z[\bX][[q]][q^{-1}] = \Z[\bX] \otimes_\Z \Z[[q]][q^{-1}]$.  For $\lambda \in \bX$, we denote by $e^\lambda$ the corresponding element of $\Z[\bX]$ or $\Z[\bX][[q]][q^{-1}]$.  If $V = \bigoplus_{n \in \Z} V_n$ is a graded $T$-representation (or a representation of some larger group, such as $B$ or $G$) with $\dim V_n < \infty$ for all $n$ and $V_n = 0$ for $n \ll 0$, we put
\[
\ch V = \sum_{n \in \Z} \sum_{\nu \in \bX} (\dim V^\nu_n) q^n e^\nu 
\]
where $V^\nu_n$ is the $\nu$-weight space of $V_n$.  More generally, if $V$ is a chain complex of graded representations, we put
\[
\ch V = \sum_{i \in \Z} (-1)^i \ch H^i(V).
\]
Next, for any $\lambda \in \bXp$, we put
\[
\chi(\lambda) = \ch H^0(\lambda) = \frac{\sum_{w \in W} (-1)^{\ell(w)} e^{w(\lambda+\rho)-\rho}}{\prod_{\alpha \in \Phi^+}(1 - e^{-\alpha})}
\qquad\text{for $\lambda \in \bXp$.}
\]
(The right-hand side is, of course, the Weyl character formula.)

Let $\lambda \in \bXp$ and $\mu \in \bX$, and let $M_\lambda^\mu(q)$ be Lusztig's $q$-analogue of the weight multiplicity.  Recall (see~\cite[(9.4)]{lus:scfqa} or~\cite[(3.3)]{bry:lws}) that this is given by
\[
M_\lambda^\mu(q) = \sum_{w \in W} (-1)^{\ell(w)} P_{w(\lambda+\rho) - (\mu + \rho)}(q),
\]
where $P_\nu(q)$ is a $q$-analogue of Kostant's partition function, determined by
\[
\prod_{\alpha \in \Phi^+} \frac{1}{1 - q e^\alpha} =
\sum_{\nu \in \bX} P_\nu(q) e^\nu.
\]
Kostant's multiplicity formula says that $M_\lambda^\mu(1)$ is the dimension of the $\mu$-weight space of the dual Weyl module $H^0(\lambda)$.

It is clear from the definition that $P_\nu(q) = 0$ unless $\nu \succeq 0$, and of course $P_0(q) = 1$.  From this, one can deduce that
\begin{equation}\label{eqn:qan-van}
M_\lambda^\mu(q) = 0 \qquad\text{if $\mu \not \preceq \lambda$.}
\end{equation}
It is known that when $\mu \in \bXp$, all coefficients in $M_\lambda^\mu(q)$ are nonnegative, but this is not true for general $\mu$.  On the other hand, for nondominant $\mu$, it may happen that $M_\lambda^\mu(q)$ is nonzero but $M_\lambda^\mu(1) = 0$.

\begin{lem}[{\cite[Lemma~6.1]{bry:lws}}]\label{lem:hesselink}
Let $\lambda\in \bX$.  We have
\[
\ch A_\lambda = \sum_{\mu \in \bXp,\ \mu \succeq \lambda} M_\mu^\lambda(q) \chi(\mu).
\]
\end{lem}

The proof of this lemma in~\cite{bry:lws} seems to assume that $\bk = \C$, but  this actually plays no role in the proof.  

\begin{thm}
Let $\lambda,\mu \in \bXp$.  As a $G$-module, $A_\lambda$ has a good filtration, and
\[
\sum_{n \ge 0} [A_\lambda : H^0(\mu)\la -2n\ra]q^n = M_\mu^\lambda(q).
\]
\end{thm}
\begin{proof}
For dominant $\lambda$, recall that $A_\lambda$ is actually a coherent sheaf on $\cN$.  The fact that $A_\lambda$ has a good filtration is due to~\cite{klt:fscb}.  The character of any $G$-module with a good filtration is, of course, a linear combination of various $\chi(\mu)$ (with $\mu \in \bXp$), and the coefficient of $\chi(\mu)$ is the multiplicity of $H^0(\mu)$. 
\end{proof}

If $M$ is a $G$-module with a good filtration, then we also have
\[
\dim \Hom_G(V(\mu),M) = [M : H^0(\mu)].
\]
This observation can be used to reformulate the preceeding theorem: for $\lambda, \mu \in \bXp$,
\[
M_\mu^\lambda(q) = \sum_{n \ge 0} \dim \Hom_{G \times \Gm}(V(\mu)\la -2n\ra, A_\lambda)q^n.
\]
This can be generalized to arbitrary $\lambda \in \bX$, using Lemma~\ref{lem:hesselink} and the fact that $\ch$ gives an embedding of the Grothendieck group of $\PCoh^{G \times \Gm}(\cN)$ in $\Z[\bX][[q]][q^{-1}]$.

\begin{thm}
Let $\lambda \in \bX$.  For any $\mu \in \bXp$, we have
\[
\sum_{i \ge 0} (-1)^i \sum_{n \ge 0} \dim \Ext^i_{G \times \Gm}(V(\mu)\la -2n\ra, A_\lambda) q^n = M_\mu^\lambda(q).
\]
\end{thm}

S.~Riche has communicated to me another proof of this fact, based on Broer's treatment~\cite{bro:sge} of the $M_\mu^\lambda(q)$ rather than Brylinski's.

We conclude this subsection with a sketch of the proof of Proposition~\ref{prop:lv-zero}. We begin with a lemma about characters of Andersen--Jantzen sheaves.

\begin{lem}\label{lem:hnabla-nreg}
Let $\lambda \in \bXp$.  We have
\[
\chi(\lambda) = \left.\left(\sum_{w \in W} (-1)^{\ell(w)} \ch A_{\lambda + \rho - w\rho}\right)\right|_{q=1}.
\]
\end{lem}
\begin{proof}
Fix a dominant weight $\mu \in \bXp$, and consider the following calculation:
\begin{multline*}
\sum_{w \in W} (-1)^{\ell(w)} M_\mu^{\lambda+\rho-w\rho}(q)
= \sum_{w,v \in W} (-1)^{\ell(w)} (-1)^{\ell(v)} P_{v(\mu+\rho) - (\lambda+2\rho-w\rho)}(q) \\
= \sum_{v \in W} (-1)^{\ell(v)} \left(\sum_{w \in W} (-1)^{\ell(w)} P_{w\rho - (\lambda + \rho - v(\mu+ \rho) + \rho)}(q) \right) \\
= \sum_{v \in W} (-1)^{\ell(v)} M_0^{\lambda+\rho - v(\mu+\rho)}(q).
\end{multline*}
Now evaluate this at $q = 1$.  We have $M_0^{\lambda+\rho- v(\mu+\rho)}(1) = 0$ unless $\lambda+\rho - v(\mu+\rho) = 0$.  But since $\lambda$ and $\mu$ are both dominant, $\lambda+\rho$ and $\mu+\rho$ are both dominant regular, and the condition $\lambda+\rho - v(\mu+\rho) = 0$ implies that $v = 1$ and $\mu = \lambda$.  Thus,
\[
\left.\left(\sum_{w \in W} (-1)^{\ell(w)} M_\mu^{\lambda+\rho-w\rho}(q)\right)\right|_{q=1} =
\begin{cases}
1 & \text{if $\mu = \lambda$,} \\
0 & \text{otherwise.}
\end{cases}
\]
The left-hand side is the coefficient of $\chi(\mu)$ in $\left(\sum (-1)^{\ell(w)} \ch A_{\lambda+\rho - w\rho}\right)\big|_{q=1}$.
\end{proof}

\begin{proof}[Proof sketch for Proposition~\ref{prop:lv-zero}]
We first describe a way to interpret ``$(\ch \fIC_\mu)|_{q=1}$'' for arbitrary $\mu \in \bXp$. Although there are typically infinitely many $q^n$ with nonzero coefficient in $\ch \fIC_\mu$, it can be shown that there is a (possibly infinite) sum 
\[
\ch \fIC_\mu = \sum_{\nu \in \bXp} c_\nu(q) \chi(\nu)
\]
where each $c_\nu(q)$ is a Laurent polynomial in $\Z[q,q^{-1}]$.  The collection of integers $\{ c_\nu(1) \}_{\nu \in \bXp}$ can be regarded as a function $\bXp \to \Z$.  In an abuse of notation, we let $(\ch \fIC_\mu)|_{q=1}$ denote that function.

A key point is that in the space of functions $\bXp \to \Z$, the various $\{ (\ch \fIC_\mu)|_{q=1} \}$ remain linearly independent.  (This fact was explained to me in 1999 by David Vogan.  It is closely related to the ideas in~\cite[Lecture 8]{vog:mco}.)

Since $\fIC_{\lambda+2\rho}\la \ell(w_0)\ra$ occurs as a composition factor in  $A_{\lambda+2\rho}$ but not in any $A_{\lambda + \rho - w\rho}$ with $w \ne w_0$, Lemma~\ref{lem:hnabla-nreg} implies that for some integers $a_\mu$, we have
\begin{equation}\label{eqn:lv-zero-calc}
\chi(\lambda) = (-1)^{\ell(w_0)} (\ch \fIC_{\lambda+2\rho})|_{q=1} + \sum_{\mu < \lambda + 2\rho} a_\mu (\ch \fIC_\mu)|_{q=1}.
\end{equation}
On the other hand, any simple $G$-representation $L(\nu)$ gives rise to a coherent intersection cohomology complex $\fIC(C_0,L(\nu))$. For some $b_\mu \in \Z$, we have
\begin{equation}\label{eqn:lv-zero-calc2}
\chi(\lambda) = (-1)^{\ell(w_0)} (\ch \fIC(C_0,L(\lambda)))|_{q=1} +
\sum_{\mu < \lambda} b_\mu (\ch \fIC(C_0,L(\mu)))|_{q=1}.
\end{equation}
An induction argument comparing~\eqref{eqn:lv-zero-calc} and~\eqref{eqn:lv-zero-calc2} yields the result.
\end{proof}

\subsection{Socles and morphisms}
\label{ss:socles}

In this subsection, we study the socles of standard objects, the cosocles of costandard objects, and $\Hom$-spaces between them.  The results for $\PCoh^{G \times \Gm}(\cN)$ strongly resemble classical facts about category $\cO$ for a complex semisimple Lie algebra, or about perverse sheaves on a flag variety (see, for instance, ~\cite[\S2.1]{bbm:te}).  In the case of $\ExCoh^{G \times \Gm}(\tcN)$, the corresponding picture is partly conjectural.

\begin{prop}\label{prop:pcoh-socle}
Let $\lambda \in \bXp$.
\begin{enumerate}
\item The socle of $\bDelta_\lambda$ is isomorphic to $\fIC_{\smp(\lambda)}\la -(2\rho^\vee,\lambda)+\delta^*_\lambda\ra$, and the cokernel of $\fIC_{\smp(\lambda)}\la -(2\rho^\vee,\lambda)+\delta^*_\lambda\ra \hookrightarrow \bDelta_\lambda$ contains no composition factor of the form $\fIC_{\mu}\la m\ra$ with $\mu \in \bXmin$.
\item The cosocle of $\bnabla_\lambda$ is isomorphic to $\fIC_{\smp(\lambda)}\la (2\rho^\vee,\lambda)-\delta^*_\lambda\ra$, and the kernel of $\bnabla_\lambda \twoheadrightarrow \fIC_{\smp(\lambda)}\la (2\rho^\vee,\lambda)-\delta^*_\lambda\ra$ contains no composition factor of the form $\fIC_{\mu}\la m\ra$ with $\mu \in \bXmin$.
\end{enumerate}
\end{prop}
\begin{proof}
Because $\bnabla_\lambda$ is a coherent sheaf, the local description of $\PCoh^{G \times \Gm}(\cN)$ from Section~\ref{ss:local} implies that it has no quotient supported on $\cN \smallsetminus \cN_\reg$.  (See~\cite[Lemma~6]{bez:pc} or~\cite[Lemma~4.1]{ab:pcs} for details.)  Therefore, its cosocle must contain only  composition factors of the form $\fIC(\cN_\reg,\cE)$.  The claims about $\bnabla_\lambda$ then follow from Propositions~\ref{prop:lv-regular} and~\ref{prop:aj-regular}. Finally, we apply Serre--Grothendieck duality to deduce the claims about $\bDelta_\lambda$ .
\end{proof}

\begin{lem}\label{lem:hom-line}
Let $\lambda, \mu \in \bX$.  We have
\[
\dim \Hom(\cO_\tcN(\lambda), \cO_\tcN(\mu)\la n\ra)
=
\begin{cases}
1 & \text{if $\lambda \succeq \mu$ and $n = (2\rho^\vee, \lambda - \mu)$,} \\
0 & \text{otherwise.}
\end{cases}
\]
\end{lem}
\begin{proof}
We have already seen in~\eqref{eqn:line-excep} that this $\Hom$-group vanishes unless $\lambda \succeq \mu$.  Assume henceforth that $\lambda \succeq \mu$.  We may also assume without loss of generality that $\mu = 0$.  Because $\cO_\tcN$ is a torsion-free coherent sheaf, the restriction map
\begin{equation}\label{eqn:hom-line}
\Hom(\cO_\tcN(\lambda), \cO_\tcN\la n\ra) \to \Hom(\cO_\tcN(\lambda)|_{\tcN_\reg}, \cO_\tcN\la n\ra|_{\tcN_\reg})
\end{equation}
is injective. The latter is a $\Hom$-group between two equivariant line bundles on a $(G \times \Gm)$-orbit.  This group has dimension $1$ if those line bundles are isomorphic, and $0$ otherwise.  In particular, $\Hom(\cO_\tcN(\lambda)|_{\tcN_\reg}, \cO_\tcN\la n\ra|_{\tcN_\reg})$ can be nonzero for at most one value of $n$, and hence likewise for $\Hom(\cO_\tcN(\lambda), \cO_\tcN\la n\ra)$.

Note that if $\Hom(\cO_\tcN(\lambda_1), \cO_\tcN\la n_1\ra)$ and $\Hom(\cO_\tcN(\lambda_2), \cO_\tcN\la n_2\ra)$ are known to be nonzero (and hence $1$-dimensional), then taking their tensor product shows that
\[
\Hom(\cO_\tcN(\lambda_1+\lambda_2), \cO_\tcN\la n_1+n_2\ra)
\]
is nonzero.  Therefore, we can reduce to the case where $\lambda$ is a simple positive root, say $\alpha$.  Note that $(2\rho^\vee, \alpha) = 2$.  Thus, to finish the proof, it suffices to exhibit a nonzero map $\cO_\tcN(\alpha) \to \cO_\tcN\la2\ra$.  We have seen such a map in~\eqref{eqn:nalpha-pres}.
\end{proof}

\begin{thm}\label{thm:hom-bnabla}
Let $\lambda, \mu \in \bXp$.  We have
\[
\dim \Hom(\bnabla_\lambda, \bnabla_\mu\la n\ra) = 
\begin{cases}
1 & \text{if $\lambda \ge \mu$ and $n = (2\rho^\vee, \lambda-\mu) - \delta^*_\lambda + \delta^*_\mu$,} \\
0 & \text{otherwise.}
\end{cases}
\]
\end{thm}
\begin{proof}
It is clear that this $\Hom$-group vanishes if $\lambda \not \ge \mu$.  If $\lambda \ge \mu$ but $n \neq (2\rho^\vee, \lambda - \mu) - \delta^*_\lambda + \delta^*_\mu$, then by Proposition~\ref{prop:pcoh-socle}, $\bnabla_\mu\la n\ra$ has no composition factor isomorphic to the cosocle of $\bnabla_\lambda$, and again the $\Hom$-group vanishes.

Assume henceforth that $\lambda \ge \mu$ and $n = (2\rho^\vee, \lambda-\mu) - \delta^*_\lambda + \delta^*_\mu$.  Let $\cK$ be the kernel of the map $\bnabla_\mu\la n\ra \to \fE_{\smp(\mu)}\la (2\rho^\vee, \lambda) - \delta^*_\lambda\ra$, and consider the exact sequence
\[
\cdots \to \Hom(\bnabla_\lambda, \cK) \to \Hom(\bnabla_\lambda, \bnabla_\mu\la n\ra) 
\overset{c}{\to} \Hom(\bnabla_\lambda, \fE_{\smp(\mu)}\la (2\rho^\vee, \lambda) - \delta^*_\lambda\ra) \to \cdots.
\]
The first term vanishes because $\cK$ contains no composition factor isomorphic to the cosocle of $\bnabla_\lambda$.  Therefore, the map labeled $c$ is injective.  The last term clearly has dimension~$1$, so $\dim \Hom(\bnabla_\lambda, \bnabla_\mu\la n\ra) \le 1$.  To finish the proof, it suffices to show that $\Hom(\bnabla_\lambda, \bnabla_\mu\la n\ra) \ne 0$.

By Lemma~\ref{lem:hom-line}, there is a nonzero map $\cO_\tcN(\lambda)\la-\delta^*_\lambda\ra \to \cO_\tcN(\mu)\la n - \delta^*_\mu\ra$.  Recall from~\eqref{eqn:hom-line} that that map has nonzero restriction to $\tcN_\reg$.  Applying $\pi_*$, we obtain a map $\bnabla_\lambda \to \bnabla_\mu\la n\ra$ that is nonzero, because its restriction to $\cN_\reg$ is nonzero.
\end{proof}

It is likely that statements of a similar flavor hold in the exotic case.  Corollary~\ref{cor:aj-regular} lets us predict what the socles of standard objects and cosocles of costandard objects should look like.  In Section~\ref{sect:sl2}, we will confirm the following statement for $G = \SL_2$.

\begin{conj}\label{conj:exo-socle}
Let $\lambda \in \bX$.
\begin{enumerate}
\item The socle of $\hDelta_\lambda$ is isomorphic to $\fE_{\smm(\lambda)}\la -(2\rho^\vee,\dom(\lambda))+\delta_\lambda\ra$, and the cokernel of $\fE_{\smm(\lambda)}\la -(2\rho^\vee,\dom(\lambda))+\delta_\lambda\ra \hookrightarrow \hDelta_\lambda$ contains no composition factor of the form $\fE_{\mu}\la m\ra$ with $\mu \in -\bXmin$.
\item The cosocle of $\hnabla_\lambda$ is isomorphic to $\fE_{\smm(\lambda)}\la (2\rho^\vee,\dom(\lambda))-\delta_\lambda\ra$, and the kernel of $\hnabla_\lambda \twoheadrightarrow \fE_{\smm(\lambda)}\la (2\rho^\vee,\dom(\lambda))-\delta_\lambda\ra$ contains no composition factor of the form $\fE_{\mu}\la m\ra$ with $\mu \in -\bXmin$.
\end{enumerate}
\end{conj}

It may be possible to prove this conjecture using the affine braid group technology developed in~\cite{br:abga, mr:etsps}.  Below is an outline of another possible approach:
\begin{enumerate}
\item Consider the pair of functors
\[
\xymatrix{
\Db\Coh^{G \times \Gm}(\tcN) \ar@<0.5ex>[r]^{\Pi_\alpha} &
\Db\Coh^{G \times \Gm}(\tcN_\alpha) \ar@<0.5ex>[l]^{\Pi^\alpha} }
\]
given by $\Pi_\alpha(\cF) = \pi_{\alpha*}i_\alpha^*(\cF(\hrho - \alpha))$ and $\Pi^\alpha(\cF) = (i_{\alpha*}\pi_\alpha^*\cF)(-\hrho)\la 1\ra$, where $\hrho$ is as in Section~\ref{ss:mu-exact}.  Note that $\Psi_\alpha \cong \Pi^\alpha \circ \Pi_\alpha$.  Check that $\Pi_\alpha$ is left adjoint to $\Pi^\alpha\la -1\ra[1]$ and right adjoint to $\Pi^\alpha\la 1\ra[1]$.

\item Define an ``exotic $t$-structure'' on $\Db\Coh^{G \times \Gm}(\tcN_\alpha)$.  Its heart should be a graded quasihereditary category whose standard (resp.~costandard) objects are $\Pi_\alpha(\hDelta_\lambda)$ (resp.~$\Pi_\alpha(\hnabla_\lambda)$) with $\lambda \prec s_\alpha\lambda$.  The functor $\Pi^\alpha$ should be $t$-exact.

\item Now imitate the strategy of~\cite[\S2.1]{bbm:te} or~\cite[Lemma~4.4.7]{by:kdkmg}, with the functors $\Pi_\alpha$ and $\Pi^\alpha$ playing the role of push-forward or pullback along the projection from the full flag variety to a partial flag variety associated to a simple root.
\end{enumerate}
One would likely have to show along the way that the distinguished triangles of Proposition~\ref{prop:braid}\eqref{it:braid-ses} are actually short exact sequences in $\ExCoh^{G \times \Gm}(\tcN)$:
\begin{equation}\label{eqn:braid-ses}
\begin{gathered}
0 \to \hDelta_{s_\alpha\lambda} \to \hDelta_\lambda\la 1\ra \to \Psi_\alpha(\hDelta_\lambda)[1] \to 0, \\
0 \to \Psi_\alpha(\hnabla_\lambda)[-1] \to \hnabla_\lambda\la -1\ra \to \hnabla_{s_\alpha\lambda} \to 0
\end{gathered}
\qquad\text{if $s_\alpha\lambda \prec \lambda$.}
\end{equation}
There should also be an equivalence like that in Theorem~\ref{thm:excoh-iw} relating $\ExCoh^G(\tcN_\alpha)$ to Iwahori--Whittaker sheaves on a partial affine flag variety $\Gv_\bK/J_\alpha$, where $J_\alpha \subset \Gv_\bO$ is the parahoric subgroup corresponding to $\alpha$.

If these expectations hold, we would obtain the following analogue of Theorem~\ref{thm:hom-bnabla}.

\begin{thm}\label{thm:hom-hnabla}
Assume that Conjecture~\ref{conj:exo-socle} holds, and that the sequences in~\eqref{eqn:braid-ses} are exact.  Let $\lambda, \mu \in \bX$.  Then $\dim \Hom(\hnabla_\lambda, \hnabla_\mu\la n\ra) \le 1$, and
\[
\dim \Hom(\hnabla_\lambda, \hnabla_\mu\la n\ra) = 0
\quad\text{if $\lambda \not \ge \mu$, or if $n \neq (2\rho^\vee,\dom(\lambda) - \dom(\mu)) - \delta_\lambda + \delta_\mu$}.
\]
If $\lambda \in \bXp$ and $\lambda \ge \mu$, then $\dim \Hom(\hnabla_\lambda, \hnabla_\mu\la (2\rho^\vee, \lambda - \dom(\mu)) + \delta_\mu\ra) = 1$.
\end{thm}
In contrast with Theorem~\ref{thm:hom-bnabla}, we do not expect the $\Hom$-group to be nonzero for arbitrary weights $\lambda \ge \mu$.  Rather, it should only be nonzero when $\mu$ is smaller than $\lambda$ in the finer partial order coming from the geometry of $\Fl$ or $\Gr$. See, for instance,~\cite[Footnote~5]{bez:ctm}.

\begin{proof}
To show that this $\Hom$-group vanishes if $\lambda \not\ge \mu$ or $n \neq (2\rho^\vee, \dom(\lambda) - \dom(\mu)) - \delta_\lambda + \delta_\mu$, and that it always has dimension at most~$1$, one can repeat the arguments from the proof of Theorem~\ref{thm:hom-bnabla}.

Suppose now that $\lambda \in \bXp$, $\lambda \ge \mu$, and $n = (2\rho^\vee, \lambda - \dom(\mu)) + \delta_\mu$.  We must show that $\Hom(\hnabla_\lambda, \hnabla_\mu\la n\ra) \ne 0$.  If $\mu$ happens to be dominant as well, then the claim follows from Lemma~\ref{lem:hom-line}.  Otherwise, note that $\lambda \ge \dom(\mu) \ge \mu$.  By the previous case, we have a nonzero map $u: \hnabla_\lambda \to \hnabla_{\dom(\mu)}\la (2\rho^\vee, \lambda - \dom(\mu))\ra$.  That map must be surjective, as can be seen by considering cosocles.  Next, the exact sequences in~\eqref{eqn:braid-ses} imply that there is a surjective map $v: \hnabla_{\dom(\mu)} \to \hnabla_\mu \la \delta_\mu\ra$.  The composition $v\la  (2\rho^\vee, \lambda - \dom(\mu))\ra \circ u$ is the desired nonzero map $\hnabla_\lambda \to \hnabla_\mu\la n\ra$.
\end{proof}

\section{Explicit computations for $\SL_2$}
\label{sect:sl2}

For the remainder of the paper, we focus on $G = \SL_2$.  In keeping with the assumptions of Section~\ref{ss:notation}, we assume that the characteristic of $\bk$ is not $2$.  We identify $\bX = \Z$ and $\bXp = \Z_{\ge 0}$.  Note that neither of the partial orders of Section~\ref{ss:notation} agrees with the usual  order on $\Z$.  In this section, $\le$ will mean the usual order on $\Z$.  We write $\preceq_\bX$ and $\le_\bX$ for those from Section~\ref{ss:notation}.  Thus, for $n, m \in \Z$, we have
\begin{center}
\begin{tabular}{ll}
$n \preceq_\bX m$ & if $m -n \in 2\Z_{\ge 0}$, \\
$n \leq_\bX m$ & if $|n| < |m|$, or else if $|n| = |m|$ and $n \leq m$.
\end{tabular}
\end{center}

\subsection{Standard and costandard exotic sheaves}

Throughout, we will work in terms of the left-hand side of the equivalence~\eqref{eqn:ind-equiv}.  Typically, ``writing down an object of $\Coh^{G \times \Gm}(\tcN)$'' will mean writing down the underlying graded $B$-module for an object of $\Coh^{B \times \Gm}(\fu)$.  For instance, the structure sheaf $\cO_\tcN$ looks like
\[
\begin{array}{rccccccccccccc}
\text{grading degree:} & \cdots & -2 & -1 & 0 & 1 & 2 & 3 & 4 & 5 & 6 & 7 & 8 & \cdots\\
\text{$B$-representation:} & \cdots & - & - & \bk_0 & - & \bk_2 & - & \bk_4 & - & \bk_6 & - & \bk_8 & \cdots
\end{array}
\]
Of course, an indecomposable object of $\Coh^{B \times \Gm}(\fu)$ must be concentrated either in even degrees or in odd degrees.  In the computations below, we will often omit grading labels for degrees in which the given module vanishes.

We will also make use of notation from Section~\ref{ss:mu-exact} such as $\tcN_\alpha$, $P_\alpha$, etc., where $\alpha = 2$ is the unique positive root of $G$. Note that $\tcN_\alpha$ can be identified with the zero section $G/B \subset \tcN$.  As in~\eqref{eqn:ind-equiv},  we have an equivalence
\[
\Coh^{B \times \Gm}(\pt) \cong \Coh^{G \times \Gm}(\tcN_\alpha).
\]
The composition of $\pi_{\alpha*}: \Db\Coh^{G \times \Gm}(\tcN_\alpha) \to \Db\Coh^{G \times \Gm}(\pt)$ with this equivalence is the induction functor $R\ind_B^G: \Db\Coh^{B \times \Gm}(\pt) \to \Db\Coh^{G \times \Gm}(\pt)$.  

If $V$ is a $B$-representation, then $i_{\alpha*}V$ denotes the object
\[
i_{\alpha*}V \cong
\begin{array}{ccccccc}
0 & 2 & 4 & 6 & 8 & 10 & \cdots\\
V & - & - & - & - & - & \cdots
\end{array}
\]
in $\Coh^{B \times \Gm}(\fu)$.  In this section, we will generally suppress the notation for $\res^G_B$ and tensor products.  For instance, in the following statement, $H^0(-n-1)\bk_{-1}$ should be understood as the $B$-representation $\res^G_B H^0(-n-1) \otimes \bk_{-1}$.

\begin{lem}\label{lem:sl2-psi}
If $n < 0$, then
\begin{multline*}
\Psi_\alpha(\cO_\tcN(n)) \cong i_{\alpha*}(V(-n-1)\bk_{-1})\la 1\ra[-1] \\
\cong
\left(\begin{array}{ccccccc}
-1 & 1 & 3 & 5 & 7 & 9 & \cdots\\
V(-n-1)\bk_{-1} & - & - & - & - & - & \cdots
\end{array}\right)[-1].
\end{multline*}
If $n > 0$, then
\[
\Psi_\alpha(\cO_\tcN(n)) \cong i_{\alpha*}(H^0(n-1)\bk_{-1})\la 1\ra \cong
\begin{array}{ccccccc}
-1 & 1 & 3 & 5 & 7 & 9 & \cdots\\
H^0(n-1)\bk_{-1} & - & - & - & - & - & \cdots
\end{array}.
\]
Finally, $\Psi_\alpha(\cO_\tcN) = 0$.
\end{lem}
\begin{proof}
Recall that $\Psi_\alpha(\cO_\tcN(n)) \cong i_{\alpha*}\pi_\alpha^* \pi_{\alpha*} i_\alpha^*(\cO_\tcN(n-1)) \otimes \cO_\tcN(-1)\la 1\ra$.  In particular, we have
\[
\pi_{\alpha*} i_\alpha^*(\cO_\tcN(n-1))  \cong \pi_{\alpha*} \cO_{G/B}(n-1) \cong R\ind_B^G \bk_{n-1}
\cong
\begin{cases}
H^0(n-1) & \text{if $n > 0$,} \\
V(-n-1)[-1] & \text{if $n < 0$,} \\
0 & \text{if $n = 0$.}
\end{cases}
\]
The result follows.
\end{proof}

\begin{prop}[Costandard exotic sheaves]\label{prop:sl2-costd}
If $n < 0$, then
\[
\hnabla_n \cong
\begin{array}{ccccccll}
-1 & 1 & 3 & 5 & 7 & 9 & \cdots\\
H^0(-n-1)\bk_{-1} & \bk_{-n} & \bk_{-n+2} & \bk_{-n+4} & \bk_{-n+6} & \bk_{-n+8} & \cdots
\end{array}
\]
If $n \ge 0$, then
\[
\hnabla_n \cong 
\begin{array}{ccccccc}
0 & 2 & 4 & 6 & 8 & 10 & \cdots\\
\bk_n & \bk_{n+2} & \bk_{n+4} & \bk_{n+6} & \bk_{n+8} & \bk_{n+10} & \cdots
\end{array}
\]
\end{prop}
\begin{proof}
For dominant weights $n \ge 0$, this is just a restatement of the fact from~\eqref{eqn:excoh-dom} that $\hnabla_n \cong \cO_\tcN(n)$.  Suppose now that $n < 0$, and consider the distinguished triangle $\hnabla_{-n}\la -1\ra \to \hnabla_n \to \Psi_\alpha(\hnabla_{-n}) \to$ from Proposition~\ref{prop:braid}\eqref{it:braid-ses}.  We have already determined the first term, and the last term is given in Lemma~\ref{lem:sl2-psi}.  Combining those gives the result.
\end{proof}

\begin{prop}[Standard exotic sheaves]\label{prop:sl2-std}
If $n \le 0$, then $\hDelta_n$ is a coherent sheaf, given by
\begin{align*}
\hDelta_0 &\cong
\begin{array}{ccccccll}
0 & 2 & 4 & 6 & 8 & 10 & \cdots\\
\bk_0 & \bk_{2} & \bk_{4} & \bk_{6} & \bk_{8} & \bk_{10} & \cdots
\end{array} \\
\hDelta_n &\cong
\begin{array}{ccccccll}
-1 & 1 & 3 & 5 & 7 & 9 & \cdots\\
\bk_n & \bk_{n+2} & \bk_{n+4} & \bk_{n+6} & \bk_{n+8} & \bk_{n+10} & \cdots
\end{array}
\qquad\text{if $n < 0$.}
\end{align*}
If $n > 0$, there is a distinguished triangle $\cH^0(\hDelta_n) \to \hDelta_n \to \cH^1(\hDelta_n)[-1] \to$ with
\[
\begin{array}{rccccccl}
& 
-2 & 0 & 2 & 4 & 6 & 8 & \cdots \\
\cH^1(\hDelta_n) \cong{} & 
V(n-1)\bk_{-1} & - & - & - & - & - & \cdots \\
\cH^0(\hDelta_n) \cong{} & 
- & \bk_{-n} & \bk_{-n+2} & \bk_{-n+4} & \bk_{-n+6} & \bk_{-n+8} & \cdots
\end{array}
\]
\end{prop}
\begin{proof}
For $n \le 0$, this is again just a restatement of~\eqref{eqn:excoh-dom}, while for $n > 0$, it follows from the distinguished triangle $\hDelta_{-n}\la -1\ra \to \hDelta_n \to \Psi(\hDelta_{-n}) \to$ of Proposition~\ref{prop:braid}\eqref{it:braid-ses}.
\end{proof}

\subsection{Auxiliary calculations}

In this subsection, we collect a number of minor results that will be needed later for the study of simple and tilting objects.

\begin{lem}\label{lem:v1}
For any $V, V' \in \Rep(G)$, we have $\RHom_{\Rep(B)}(V, V'\bk_{-1}) = 0$.
\end{lem}
\begin{proof}
By adjunction, $\RHom(V, V'\bk_{-1}) \cong \RHom_{\Rep(G)}(V, R\ind_B^G(V'\bk_{-1}))$. But $R\ind_B^G (V'\bk_{-1}) \cong V' \otimes R\ind_B^G \bk_{-1} = 0$.
\end{proof}

\begin{lem}\label{lem:v2}
For any $V \in \Rep(G)$, $i_{\alpha*}(V\bk_{-2})$ lies in $\ExCoh^{G \times \Gm}(\tcN)$.
\end{lem}
\begin{proof}
It suffices to show that 
\[
\uHom(\hDelta_m[-k], i_{\alpha*}V\bk_{-2}) = 
\uHom(i_{\alpha*}V\bk_{-2}, \hnabla_m[k]) = 0
\qquad\text{for all $k < 0$.}
\]
The vanishing of the latter is obvious, since $i_{\alpha*}V\bk_{-2}$ and $\hnabla_m$ are both coherent sheaves.  Likewise, the vanishing of the former is obvious when $m \le 0$, or when $k < -1$.  When $m > 0$ and $k = -1$, using Lemma~\ref{lem:v1}, we have
\begin{multline*}
\uHom(\hDelta_m[1], i_{\alpha*}V\bk_{-2}) \cong
\uHom(i_{\alpha*}V(m-1)\bk_{-1}\la 2\ra, i_{\alpha*}V\bk_{-2}) \\
\cong \Hom_{\Rep(B)}(V(m-1)\bk_{-1}, V\bk_{-2}) \la -2\ra  \\
\cong \Hom_{\Rep(B)}(V(m-1), V\bk_{-1})\la -2\ra = 0. \qedhere
\end{multline*}
\end{proof}

\begin{lem}\label{lem:brep}
For $n > 0$, there are short exact sequences of $B$-representations
\begin{gather*}
0 \to H^0(n-1)\bk_{-1} \to H^0(n) \to \bk_n \to 0, \\
0 \to \bk_{-n} \to V(n) \to V(n-1)\bk_1 \to 0.
\end{gather*}
\end{lem}
\begin{proof}
This can be checked by direct computation using, say, the realization of $H^0(n)$ as the space of homogeneous polynomials of degree $n$ on $\bA^2$.
\end{proof}

\begin{lem}\label{lem:new-costd}
For $n \le -2$, there is a short exact sequence in $\ExCoh^{G \times \Gm}(\tcN)$:
\[
0 \to i_{\alpha*}H^0(-n-2)\bk_{-2}\la 1\ra \to \hnabla_n \to \hnabla_{-n-2}\la 1\ra \to 0.
\]
\end{lem}
\begin{proof}
Note that for any $G$-representation $V$, there are natural isomorphisms
\begin{multline}\label{eqn:hnc-dim}
\Hom_{\Coh^{G \times \Gm}(\tcN)}(V \otimes \cO_\tcN\la k\ra, \cO_\tcN(m)) \cong
\Hom_{\Rep(B \times \Gm)}(V\la k\ra, \bk_m \otimes \bk[u]) \\
\cong
\begin{cases}
\Hom_{\Rep(B)}(V, \bk_{m - k}) & \text{if $k \le 0$ and $k$ is even,} \\
0 & \text{otherwise.}
\end{cases}
\end{multline}
For instance, we can see in this way that
\begin{equation}\label{eqn:hnc1}
\begin{gathered}
\Hom(H^0(-n-1) \otimes \cO_\tcN(1)\la -2\ra, \cO_\tcN(-n-2)), \\
\Hom(H^0(-n-1) \otimes \cO_\tcN(-1), \cO_\tcN(-n-2))
\end{gathered}
\end{equation}
are both $1$-dimensional.  Consider the following exact sequence, induced by~\eqref{eqn:nalpha-pres}:
\begin{multline}\label{eqn:hnc-ses}
0 \to H^0(-n-1) \otimes \cO_\tcN(1)\la -2\ra \overset{i}{\to} H^0(-n-1) \otimes \cO_\tcN(-1)\\
 \to i_{\alpha*}(H^0(-n-1)\bk_{-1}) \to 0.
\end{multline}
The map $i$ induces an isomorphism between the two $\Hom$-groups in~\eqref{eqn:hnc1}.

By similar reasoning, we find that
\begin{multline}\label{eqn:hnc2}
\Ext^1(H^0(-n-1) \otimes \cO_\tcN(-1), \cO_\tcN(-n-2))
\\
\cong \Ext^1_{\Rep(B)}(H^0(-n-1)\bk_{-1}, \bk_{-n-2}) 
\cong \Ext^1_{\Rep(B)}(H^0(-n-1), \bk_{-n-1})
\\
\cong \Ext^1_{\Rep(G)}(H^0(-n-1), H^0(-n-1)) = 0.
\end{multline}
Now apply $\Ext^\bullet({-}, \cO_\tcN(-n-2))$ to~\eqref{eqn:hnc-ses} to obtain a long exact sequence.  From~\eqref{eqn:hnc1} and~\eqref{eqn:hnc2}, we deduce that
\begin{equation}\label{eqn:hnc-ext}
\Ext^1(i_{\alpha*}(H^0(-n-1)\bk_{-1}), \cO_\tcN(-n-2)) = 0.
\end{equation}

Now apply $\Hom({-}, \cO_\tcN(-n-2)\la 1\ra)$ to the distinguished triangle
\begin{equation}\label{eqn:hnc-dt}
\cO_\tcN(-n)\la -1\ra \to \hnabla_n \to i_{\alpha*}H^0(-n-1)\bk_{-1}\la 1\ra \to 
\end{equation}
from Proposition~\ref{prop:braid}\eqref{it:braid-ses}.  It is clear that $\Hom(i_{\alpha*}H^0(-n-1)\bk_{-1}\la 1\ra, \cO_\tcN(-n-2)\la 1\ra)$ vanishes.  Combining this with~\eqref{eqn:hnc-ext}, we find that
\[
\Hom(\hnabla_n, \cO_\tcN(-n-2)\la 1\ra) \simto \Hom(\cO_\tcN(-n)\la -1\ra, \cO_\tcN(-n-2)\la 1\ra)
\]
is an isomorphism.  The latter is $1$-dimensional (by~\eqref{eqn:hnc-dim}), so the former is as well.

We have constructed a nonzero map $h: \hnabla_n \to \hnabla_{-n-2}\la 1\ra$.  We claim that as a map of coherent sheaves, $h$ is surjective.  By construction, its image at least contains the image of $\hnabla_{-n}\la -1\ra \cong \cO_\tcN(-n)\la -1\ra$, i.e., the submodule containing all homogeneous elements in degrees${}\ge 1$.  The only question is whether $h$ is surjective in grading degree $-1$.  If it were not, then $\hnabla_{-n}\la -1\ra$ would be a quotient of $\hnabla_n$ as a coherent sheaf.  This would imply the splitting of the distinguished triangle~\eqref{eqn:hnc-dt}, contradicting the indecomposability of $\hnabla_n$.  Thus, $h$ is surjective.

Restricting $h$ to the space of homogeneous elements of degree $-1$, we get a surjective map of $B$-representations $H^0(-n-1)\bk_{-1} \to \bk_{-n-2}$.  Lemma~\ref{lem:brep} identifies the kernel of that map for us.  We therefore have a short exact sequence of coherent sheaves
\[
0 \to i_{\alpha*}H^0(-n-2)\bk_{-2}\la 1\ra \to \hnabla_n \to \hnabla_{-n-2}\la 1\ra \to 0.
\]
Lemma~\ref{lem:v2} tells us that this is also a short exact sequence in $\ExCoh^{G \times \Gm}(\tcN)$.
\end{proof}

\begin{lem}\label{lem:new-std}
For $n \le -2$, there is a short exact sequence in $\ExCoh^{G \times \Gm}(\tcN)$:
\[
0 \to \hDelta_{-n-2}\la -1\ra \to \hDelta_n \to i_{\alpha*}V(-n-2)\bk_{-2} \to 0.
\]
\end{lem}
We omit the proof of this lemma, which is quite similar to Lemma~\ref{lem:new-costd}.  Note that in grading degree $-1$, we have the distinguished triangle of $B$-representations
\[
V(-n-3)\bk_{-1}[-1] \to \bk_n \to V(-n-2)\bk_{-2} \to
\]
obtained from Lemma~\ref{lem:brep} by tensoring with $\bk_{-2}$.

\begin{lem}\label{lem:v11}
For any $V \in \Rep(G)$, $i_{\alpha*}(V\bk_{-1})[-1]$ lies in $\ExCoh^{G \times \Gm}(\tcN)$.
\end{lem}
\begin{proof}
As in Lemma~\ref{lem:v2}, it suffices to show that 
\[
\uHom(\hDelta_m[-k], i_{\alpha*} V\bk_{-1}[-1]) = 
\uHom(i_{\alpha*} V\bk_{-1}[-1], \hnabla_m[k]) = 0
\qquad\text{for all $k < 0$.}
\]
The vanishing of the former is easily seen in terms of the natural $t$-structure on $\Db\Coh^{G \times \Gm}(\tcN)$.  The vanishing of the latter is also clear when $k \le -2$.  If $k = -1$ and $m \ge -1$, then this $\Hom$-group vanishes because $\hnabla_m$ is a torsion-free coherent sheaf, while $i_{\alpha*} V\bk_{-1}$ is torsion.  Finally, if $k = -1$ and $m \le -2$, we use Lemma~\ref{lem:new-costd}.  Consider the exact sequence
\begin{multline}\label{eqn:v11-les}
0 \to \uHom(i_{\alpha*} V\bk_{-1}, i_{\alpha*}H^0(-m-2)\bk_{-2})
\to \uHom(i_{\alpha*} V\bk_{-1}, \hnabla_m) \\
\to \uHom(i_{\alpha*} V\bk_{-1}, \hnabla_{-m-2}\la 1\ra) \to \cdots
\end{multline}
We have already seen that the last term vanishes.  The first term is isomorphic to
\[
\uHom_{\Rep(B \times \Gm)}(V\bk_{-1}, H^0(-m-2)\bk_{-2}) \cong
\uHom_{\Rep(B \times \Gm)}(V, H^0(-m-2)\bk_{-1}),
\]
and this vanishes by Lemma~\ref{lem:v1}.  So the middle term in~\eqref{eqn:v11-les} vanishes as well, and $i_{\alpha*} V\bk_{-1}\la 2\ra[-1]$ lies in $\ExCoh^{G \times \Gm}(\tcN)$, as desired.
\end{proof}

\begin{lem}\label{lem:o1-twist}
We have
\begin{align*}
\pi_*(\hDelta_n \otimes \cO_\tcN(1)) &\cong
\begin{cases}
\pi_*(\hDelta_{n+1}) & \text{if $n \le -2$,} \\
\pi_*(\hDelta_{-n-1})\la 1\ra & \text{if $n \ge -1$.}
\end{cases}
\\
\pi_*(\hnabla_n \otimes \cO_\tcN(1)) &\cong
\begin{cases}
\pi_*(\hnabla_{-n-1})\la 1\ra & \text{if $n < 0$,} \\
\pi_*(\hnabla_{n+1}) & \text{if $n \ge 0$.}
\end{cases}
\end{align*}
\end{lem}
\begin{proof}
If $n \ge 0$, then it is clear that $\hnabla_n \otimes \cO_\tcN(1) \cong \hnabla_{n+1}$.  Similarly, for $n = -1$, we have $\hnabla_{-1} \otimes \cO_\tcN(1) \cong \hnabla_0\la 1\ra$.  If $n \le -2$, we use Lemma~\ref{lem:new-costd} together with the fact that $\pi_*(i_{\alpha*} H^0(-n-2)\bk_{-1}) = 0$ to deduce that
\[
\pi_*(\hnabla_n \otimes \cO_\tcN(1)) \cong \pi_*(\hnabla_{-n-2}\la 1\ra \otimes \cO_\tcN(1)) \cong \pi_*(\hnabla_{-n-1})\la 1\ra.
\]
The proof for standard objects is similar.
\end{proof}

\subsection{Socles and morphisms}

In this subsection, we verify Conjecture~\ref{conj:exo-socle} and the conclusions of Theorem~\ref{thm:hom-hnabla} for $G = \SL_2$.

\begin{prop}\label{prop:sl2-braid-ses}
If $n > 0$, we have the following short exact sequences in $\ExCoh^{G \times \Gm}(\tcN)$:
\begin{gather*}
0 \to \hDelta_{-n} \to \hDelta_n\la 1\ra \to \Psi_\alpha(\hDelta_n)[1] \to 0, \\
0 \to \Psi_\alpha(\hnabla_n)[-1] \to \hnabla_n\la -1\ra \to \hnabla_{-n} \to 0.
\end{gather*}
\end{prop}
\begin{proof}
In view of Proposition~\ref{prop:braid}, all we need to do is check that $\Psi_\alpha(\hDelta_n)[1]$ and $\Psi_\alpha(\hnabla_n)[-1]$ lie in $\ExCoh^{G \times \Gm}(\tcN)$.  By Lemma~\ref{lem:sl2-psi}, we have
\begin{gather*}
\Psi_\alpha(\hDelta_n)[1] \cong \Psi_\alpha(\hDelta_{-n})\la 1\ra \cong \Psi_\alpha(\cO_\tcN(-n))\la 2\ra
\cong i_{\alpha*}V(n-1)\bk_{-1}\la 3\ra[-1], \\
\Psi_\alpha(\hnabla_n)[-1] \cong \Psi_\alpha(\cO_\tcN(n))[-1] \cong i_{\alpha*}H^0(n-1)\bk_{-1}\la 1\ra[-1].
\end{gather*}
By Lemma~\ref{lem:v11}, these both lie in $\ExCoh^{G \times \Gm}(\tcN)$.
\end{proof}

\begin{prop}\label{prop:sl2-costd-hom}
Let $m,n \in \Z$.  Then
\[
\dim \Hom(\hnabla_m, \hnabla_n\la k\ra) =
\begin{cases}
1 & \text{if $m \ge_{\bX} n$ and $k = |m| - |n| - \delta_m + \delta_n$,} \\
0 & \text{otherwise.}
\end{cases}
\]
Moreover, any nonzero map $\hnabla_m \to \hnabla_n\la k\ra$ is surjective.
\end{prop}
\begin{proof}
The determination of $\dim \Hom(\hnabla_m, \hnabla_n\la k\ra)$ can be done by direct computation using Proposition~\ref{prop:sl2-costd}.  Note that each costandard object is generated as a $\bk[\fu]$-module by a single homogeneous component (lying in grading degree $0$ or $-1$).  Moreover, the costandard objects associated to dominant weights are free over $\bk[\fu]$.  With these observations, the problem of computing $\Hom(\hnabla_m, \hnabla_n\la k\ra)$ can be reduced to that of computing $\Hom$-groups between certain $B$-representations.  The latter is quite straightforward.

We now consider the surjectivity claim.  Suppose $m \ge_\bX n$.  If $m \ge 0$, then
\[
m \ge_{\bX} -m \ge_{\bX} m-2 \ge_{\bX} -m+2 \ge_{\bX} \cdots \ge_{\bX} n.
\]
Lemma~\ref{lem:new-costd} and Proposition~\ref{prop:sl2-braid-ses} together give us a collection of surjective maps
\[
\hnabla_m \twoheadrightarrow
\hnabla_{-m}\la 1\ra \twoheadrightarrow 
\hnabla_{m-2}\la 2\ra \twoheadrightarrow 
\hnabla_{-m+2}\la 3\ra \twoheadrightarrow 
\cdots \twoheadrightarrow
\hnabla_n\la m - |n| + \delta_n\ra.
\]
Their composition is a nonzero element of $\Hom(\hnabla_m, \hnabla_n\la m - |n| + \delta_n\ra)$, and it is surjective.  Similar reasoning applies if $m < 0$.
\end{proof}

\begin{prop}
Let $n \in \Z$.
\begin{enumerate}
\item The socle of $\hDelta_n$ is isomorphic to $\fE_{\smm(n)}\la -|n|+ \delta_n\ra$, and the cokernel of $\fE_{\smm(n)}\la -|n|+ \delta_n\ra \hookrightarrow \hDelta_n$ contains no composition factor of the form $\fE_m\la k\ra$ with $m \in \{0,-1\}$.
\item The cosocle of $\hnabla_n$ is isomorphic to $\fE_{\smm(n)}\la |n|- \delta_n\ra$, and the kernel of $\hnabla_n \twoheadrightarrow \fE_{\smm(n)}\la |n|- \delta_n\ra $ contains no composition factor of the form $\fE_m\la k\ra$ with $m \in \{0,-1\}$.
\end{enumerate}
\end{prop}
\begin{proof}
We will treat only the costandard case.  Proposition~\ref{prop:sl2-costd-hom} tells us that there is a surjective map $\hnabla_n \twoheadrightarrow \hnabla_{\smm(n)}\la |n|- \delta_n\ra \cong \fE_{\smm(n)}\la |n|- \delta_n\ra$.  Corollary~\ref{cor:aj-regular} already tells us that $\hnabla_n$ can have no other antiminuscule composition factor.  To finish the proof, we must show that $\hnabla_n$ has no simple quotient $\fE_m\la k\ra$ with $m \notin \{0,-1\}$.  If it did, then the composition $\hnabla_n \to \fE_m\la k\ra \to \hnabla_m\la k\ra$ would be a nonzero, nonsurjective map, contradicting Proposition~\ref{prop:sl2-costd-hom}.
\end{proof}

\subsection{Simple and tilting exotic sheaves}

We are now ready to determine the simple exotic sheaves $\fE_n$ and the indecomposable tilting objects $\hfT_n$ for all $n \in \Z$.

\begin{prop}[Simple exotic sheaves]
We have
\[
\fE_n \cong
\begin{cases}
i_{\alpha*} L(-n-2)\bk_{-2}\la 1\ra & \text{if $n \le -2$,} \\
\cO_\tcN(-1)\la 1\ra & \text{if $n = -1$,} \\
\cO_\tcN & \text{if $n = 0$,} \\
i_{\alpha*} L(n-1)\bk_{-1}\la 2\ra[-1] & \text{if $n \ge 1$.}
\end{cases}
\]
\end{prop}
\begin{proof}
The weights $n = 0$ and $n = -1$ are antiminuscule, so in those cases, $\fE_n$ is given by Lemma~\ref{lem:minu-exotic}.  

By Lemmas~\ref{lem:v2} and~\ref{lem:v11}, respectively, we know that $i_{\alpha*}L(-n-2)\bk_{-2}\la 1\ra$ and $i_{\alpha*} L(n-1)\bk_{-1}\la 2\ra[-1]$ belong to $\ExCoh^{G \times \Gm}(\tcN)$.  One can show by induction with respect to the partial order $\le_{\bX}$ that they are simple, using our explicit description of the standard and costandard objects.  We omit further details.
\end{proof}

\begin{prop}[Tilting exotic sheaves]\label{prop:sl2-tilt}
We have
\[
\hfT_n \cong
\begin{cases}
T(-n-1) \otimes \cO_\tcN(-1)\la 1\ra & \text{if $n < 0$,} \\
T(n) \otimes \cO_\tcN & \text{if $n \ge 0$.}
\end{cases}
\]
\end{prop}
\begin{proof}
For $n \ge 0$, this is just a restatement of Proposition~\ref{prop:tilt-dom}.  Assume henceforth that $n < 0$. To show that the $T(-n-1) \otimes \cO_\tcN(-1)$ are tilting objects, we will use the criterion of~\cite[Lemma~4]{bez:ctm}, which says that it is enough to check that for all $k > 0$, we have
\[
\uHom(\hDelta_m[-k], T(-n-1) \otimes \cO_\tcN(-1)) =
\uHom(T(-n-1) \otimes \cO_\tcN(-1), \hnabla_m[k]) = 0,
\]
or, equivalently,
\begin{multline*}
\uHom(\hDelta_m \otimes \cO_\tcN(1)[-k], T(-n-1) \otimes \cO_\tcN) = \\
\uHom(T(-n-1) \otimes \cO_\tcN, \hnabla_m \otimes \cO_\tcN(1)[k]) = 0.
\end{multline*}
By adjunction and the fact that $\pi^*\cO_\cN \cong \pi^!\cO_\cN \cong \cO_\tcN$, this is in turn equivalent to the vanishing of the following $\Hom$-groups in $\Db\Coh^{G \times \Gm}(\cN)$:
\begin{multline*}
\uHom(\pi_*(\hDelta_m \otimes \cO_\tcN(1))[-k], T(-n-1) \otimes \cO_\cN) = \\
\uHom(T(-n-1) \otimes \cO_\cN, \pi_*(\hnabla_m \otimes \cO_\tcN(1))[k]) = 0.
\end{multline*}
These equalities hold because $T(n-1) \otimes \cO_\cN$ is a tilting object in $\PCoh^{G \times \Gm}(\cN)$ (Proposition~\ref{prop:pcoh-tilt}), while $\pi_*(\hDelta_m \otimes \cO_\tcN(1))$ is proper standard and $\pi_*(\hnabla_m \otimes \cO_\tcN(1))$ is proper costandard (by Lemma~\ref{lem:o1-twist} and Proposition~\ref{prop:mu-exact}).

There is an obvious morphism $\hDelta_n \to T(-n-1) \otimes \cO_\tcN(-1)\la 1\ra$, and this shows that $\hfT_n \cong T(-n-1) \otimes \cO_\tcN(-1)\la 1\ra$, as desired.
\end{proof}

\begin{prop}
If $\bk = \C$, then every standard or costandard object in $\ExCoh^{G \times \Gm}(\tcN)$ is uniserial.
\end{prop}
\begin{proof}
This holds by induction with respect to $\le_{\bX}$, using the short exact sequences in Lemmas~\ref{lem:new-costd} and~\ref{lem:new-std} and Proposition~\ref{prop:sl2-braid-ses}.
\end{proof}

For example, the composition series of $\hnabla_n$ looks like this:
\begin{align*}
n \ge 0: \hspace{-2em} &
\begin{array}{r|c|}
\cline{2-2}
\text{cosocle:} & \fE_{\smm(n)}\la n \ra \\ \cline{2-2}
& \vdots \\ \cline{2-2}
& \fE_{-n+2}\la 3\ra \\ \cline{2-2}
&\fE_{n-2}\la 2\ra \\ \cline{2-2}
& \fE_{-n}\la 1\ra \\ \cline{2-2}
\text{socle:} & \fE_n  \\ \cline{2-2}
\end{array}
&&&
n < 0: \hspace{-2em} &
\begin{array}{r|c|}
\cline{2-2}
\text{cosocle:} & \fE_{\smm(n)}\la -n-1\ra \\ \cline{2-2}
& \vdots \\ \cline{2-2}
& \fE_{-n-4}\la 3\ra \\ \cline{2-2}
&\fE_{n+2}\la 2\ra \\ \cline{2-2}
& \fE_{-n-2}\la 1\ra \\ \cline{2-2}
\text{socle:} & \fE_n  \\ \cline{2-2}
\end{array}.
\end{align*}

We conclude by answering Question~\ref{ques:positivity} for $G = \SL_2$.

\begin{prop}[Positivity for tilting exotic sheaves]
For any $n, m \in \Z$, the graded vector space $\uHom(\hfT_n, \hfT_m)$ is concentrated in nonnegative degrees.
\end{prop}
\begin{proof}
We must show that $\Hom(\hfT_n,\hfT_m\la k\ra) = 0$ for all $k < 0$.
From Proposition~\ref{prop:sl2-tilt}, this is obvious if $n < 0$ or if $m \ge 0$.  It is also obvious if $n \ge 0$, $m < 0$, and $k \le -2$.  It remains to consider the case where $n \ge 0$, $m < 0$, and $k = -1$.  Using Lemma~\ref{lem:v1}, we find that
\[
\Hom(T(n) \otimes \cO_\tcN, T(-m-1) \otimes \cO_\tcN(-1))
\cong \Hom_{\Rep(B)}(T(n), T(-m-1)\bk_{-1}) = 0,
\]
as desired.
\end{proof}

\subsection{Perverse-coherent sheaves}

After the hard work of the exotic case, the calculations in the perverse coherent case are relatively easy.

\begin{prop}\label{prop:sl2-pcoh}
For $n \in \Z_{\ge 0}$, $\bnabla_n$ is given by
\begin{align*}
\bnabla_0 &\cong
\begin{array}{ccccccll}
0 & 2 & 4 & 6 & 8 & 10 & \cdots\\
H^0(0) & H^0(2) & H^0(4) & H^0(6) & H^0(8) & H^0(10) & \cdots
\end{array} \\
\bnabla_n &\cong
\begin{array}{cccccc}
1 & 3 & 5 & 7 & \cdots\\
H^0(n) & H^0(n+2) & H^0(n+4) & H^0(n+6) & \cdots
\end{array}
\qquad\text{if $n > 0$.}
\end{align*}
For $n \in \{0,1\}$, we have $\bDelta_n \cong \bnabla_n$, whereas for $n \ge 2$, we have
{\tiny\[
\begin{array}{rcccccccc}
& 
  &   &        &  n - 3     &  n - 1     &  n + 1  & n + 3\\
&
-1 & 1 & \cdots & {}-\smp(n) & {}-\smp(n) & {}-\smp(n) & {}-\smp(n) & \cdots\\
\cH^1(\bDelta_n) \cong{} & 
V(n-2) & V(n-4) & \cdots & V(\smp(n)) & - & - & \cdots \\
\cH^0(\bDelta_n) \cong{} & 
- & - & \cdots & - & H^0(\smm(n)) & H^0(\smm(n)+2) & H^0(\smm(n)+4) & \cdots
\end{array}
\]}
\end{prop}
\begin{proof}
Apply $R\ind_B^G$ to the formulas from Propositions~\ref{prop:sl2-costd} and~\ref{prop:sl2-std}.
\end{proof}

With a bit more effort, it is possible to give a finer description of the $\bk[\cN]$-action on these modules.  Recall that for $\SL_2$, the nilpotent cone $\cN$ is isomorphic as an $\SL_2$-variety to the quotient $\bA^2/(\Z/2)$, where the nontrivial element of $\Z/2$ acts by negation.  This gives rise to an isomorphism
\[
\bk[\cN] \cong \bk[x^2, xy, y^2],
\]
where the right-hand side is the subring of $(\Z/2)$-invariant elements in the polynomial ring $\bk[x,y]$.  For $n \in \Z_{\ge 0}$, let 
\[
M_n = \bk[x^2, xy, y^2]\cdot (x^n, x^{n-1}y, \ldots, y^n) \subset \bk[x,y].
\]
Thus, $M_n$ consists of polynomials whose terms have degrees${}\ge n$ and${}\equiv n \pmod 2$.

\begin{lem}\label{lem:bnabla-poly}
There is an isomorphism of $\bk[\cN]$-modules $\hnabla_n \cong M_n\la n - \delta^*_n\ra$.
\end{lem}
\begin{proof}
For $n = 0$, we have $M_0 \cong \cO_\cN \cong \bnabla_0$, and there is nothing to prove.  For $n = 1$, recall that $\bnabla_1$ is a simple perverse-coherent sheaf, and up to grading shift, it is the unique simple object that is a torsion-free coherent sheaf not isomorphic to $\cO_\cN$.  It is easy to check from Proposition~\ref{prop:sl2-pcoh} that $M_1$ is isomorphic as a $(G \times \Gm)$-representation to $\bnabla_1$ (and that it is \emph{not} isomorphic to $\bnabla_0$), so to prove that $M_1 \cong \bnabla_1$, it suffices to show that $M_1$ is a simple perverse-coherent sheaf.  This can be done by computing its local cohomology at $0$, and then using the criterion described after Theorem~\ref{thm:pcoh-loc}.

For $n \ge 2$, Proposition~\ref{prop:pcoh-socle} gives us a map
\[
\bnabla_n \to \bnabla_{\smp(n)}\la n - 1\ra
\]
that is surjective as a morphism in $\PCoh^{G \times \Gm}(\cN)$.  We claim that it is also injective as a morphism in $\Coh^{G \times \Gm}(\cN)$.  Indeed, the map is an isomorphism over $\cN_\reg$, so its kernel would have to be supported on $\cN \smallsetminus \cN_\reg$.  But $\bnabla_n$ is a torsion-free coherent sheaf (see Proposition~\ref{prop:pcoh-torsion}), so that kernel must be trivial.  In other words, as a coherent sheaf, $\bnabla_n$ can be identified with a certain submodule of $\bnabla_{\smp(n)}\la n-1\ra$.  Proposition~\ref{prop:sl2-pcoh} shows us that the desired submodule is precisely $M_n$.
\end{proof}

Let $i_0: \{0\} \hookrightarrow \cN$ be the inclusion map of the origin into the nilpotent cone.

\begin{prop}[Simple perverse-coherent sheaves]
We have
\[
\fIC_n \cong
\begin{cases}
\cO_\cN & \text{if $n = 0$,} \\
\bDelta_1 \cong \bnabla_1 \cong M_1& \text{if $n = 1$,} \\
i_{0*}L(n-2)\la 1\ra[-1] & \text{if $n \ge 2$.}
\end{cases}
\]
\end{prop}
\begin{proof}
The local description of $\PCoh^{G \times \Gm}(\cN)$ from Section~\ref{ss:local} makes it clear that the list of objects above is an exhaustive list of simple perverse-coherent sheaves up to grading shift.  To check the parametrization in the cases where $n \ge 2$, we simply note that there is a nonzero map $\bDelta_n \to i_{0*}L(n-2)\la 1\ra[-1]$.
\end{proof}

Recall that the tilting objects in $\PCoh^{G \times \Gm}(\cN)$ have been completely described in Proposition~\ref{prop:pcoh-tilt}.  We will not repeat that description here.

\begin{prop}\label{prop:sl2-ses}
For any $n \ge 2$, we have the following short exact sequences in $\PCoh^{G \times \Gm}(\cN)$:
\begin{gather*}
0 \to i_{0*} H^0(n-2)\la 1\ra[-1] \to \bnabla_n \to \bnabla_{n-2}\la 1 + \delta^*_{n-2}\ra \to 0, \\
0 \to \bDelta_{n-2}\la -1-\delta^*_{n-2}\ra \to \bDelta_n \to i_{0*}V(n-2)\la 1\ra[-1] \to 0.
\end{gather*}
\end{prop}
\begin{proof}
We already know from Theorem~\ref{thm:hom-bnabla} that $\dim \Hom(\bnabla_n, \bnabla_{n-2}\la 1 + \delta^*_{n-2}\ra) = 1$.  Lemma~\ref{lem:bnabla-poly} lets us identify the cone of any such map (up to grading shift) with the space of homogeneous polynomials in $\bk[x,y]$ of degree $n-2$: in other words, with $H^0(n-2)$.  The local description of $\PCoh^{G \times \Gm}(\cN)$ implies that $i_{0*}H^0(n-2)\la 1\ra[-1]$ is indeed a perverse-coherent sheaf, and this gives us the first short exact sequence above.  The second is then obtained by applying the Serre--Grothendieck duality functor $\SGD$.
\end{proof}

\begin{prop}
If $\bk = \C$, then every standard or costandard object in $\PCoh^{G \times \Gm}(\tcN)$ is uniserial.
\end{prop}
\begin{proof}
This is immediate from Proposition~\ref{prop:sl2-ses}.
\end{proof}

For example, the composition series of $\bnabla_n$ for $n > 0$ looks like this:
\begin{align*}
\text{$n$ even and${}\ge 2$:} \hspace{-2em} &
\begin{array}{r|c|}
\cline{2-2}
\text{cosocle:} & \fIC_{0}\la n -1\ra \\ \cline{2-2}
& \fIC_{2}\la n-2\ra \\ \cline{2-2}
&\fIC_{4}\la n-4\ra \\ \cline{2-2}
& \vdots \\ \cline{2-2}
& \fIC_{n-4}\la 4\ra \\ \cline{2-2}
& \fIC_{n-2}\la 2\ra \\ \cline{2-2}
\text{socle:} & \fIC_n  \\ \cline{2-2}
\end{array}
&&&
\text{$n$ odd:} \hspace{-2em} &
\begin{array}{r|c|}
\cline{2-2}
\text{cosocle:} & \fIC_{1}\la n-1\ra \\ \cline{2-2}
& \fIC_{3}\la n-3\ra \\ \cline{2-2}
&\fIC_{5}\la n-5\ra \\ \cline{2-2}
& \vdots \\ \cline{2-2}
& \fIC_{n-4}\la 4\ra \\ \cline{2-2}
& \fIC_{n-2}\la 2\ra \\ \cline{2-2}
\text{socle:} & \fIC_n  \\ \cline{2-2}
\end{array}
\end{align*}

Finally, since $\PCoh^{G \times \Gm}(\cN)$ is properly stratified but not quasihereditary, it also has true standard and true costandard objects.  The following proposition describes them.  

\begin{prop}
We have $\Delta_0 \cong \nabla_0 \cong \cO_\cN$.  If $n > 0$, there are short exact sequences
\[
0 \to \bnabla_n \to \nabla_n \to \bnabla_n\la 2\ra \to 0,
\qquad
0 \to \bDelta_n\la -2\ra \to \Delta_n \to \bDelta_n \to 0.
\]
\end{prop}

\section*{Exercises}

\setcounter{subsection}{0}
\renewcommand{\thesubsection}{\arabic{subsection}}

\addtocontents{toc}{\protect\setcounter{tocdepth}{1}}

\subsection{}
This exercise involves the $q$-analogues discussed in Section~\ref{ss:charform} for $G = \SL_2$.  Let $n \in \Z_{\ge 0}$ and let $m \in \Z$. Show that
\[
P_n(q) = \begin{cases}
q^{n/2} & \text{if $n \in 2\Z_{\ge 0}$,} \\
0 & \text{otherwise,}
\end{cases}
\]
and
\[
M_n^m(q) = \begin{cases}
0 & \text{if $m > n$ or $n \not\equiv m \pmod 2$}, \\
q^{(n -m)/2} & \text{if $-n \le m \le n$ and $n \equiv m \pmod 2$,} \\
q^{(n-m)/2} - q^{(-n-m-2)/2} & \text{of $m \le n-2$ and $n \equiv m \pmod 2$.}
\end{cases}
\]
In particular, when $m < -n$, $M_n^m(1)$ is always zero, but $M_n^m(q)$ is not.

\subsection{}
Let $G = \SL_2$, and let $n \in \Z_{\ge 0}$. Explicitly describe the bigraded algebra $\uExt^\bullet(\bnabla_n, \bnabla_n)$.  \emph{Answer}: If $n = 0$, this algebra is just $\bk$.  If $n > 0$, this algebra is a polynomial ring with one generator living in bidegree $(1,-2)$.

\subsection{}
Let $G = \SL_2$.  Compute explicitly the $\bk[\cN]$-module structure on $\cH^1(\bDelta_n)$.

\subsection{}
(a)~Prove that for all $\lambda \in \bX$, the line bundle $\cO_\tcN(\lambda)$ lies in $\ExCoh^{G \times \Gm}(\tcN)$.  Then prove that $\fE_\lambda\la -\delta_\lambda\ra$ occurs as a composition factor of $\cO_\tcN(\lambda)$ with multiplicity~$1$, and that all other composition factors $\fE_\mu\la k\ra$ satisfy $\mu < \lambda$.

(b)~For $G = \SL_2$, compute the multiplicites of all composition factors in $\cO_\tcN(\lambda)$.

%

\subsection{}
Let $\lambda \in \bX$.  Prove that
\[
\dim \Hom(A_{w\lambda}, A_\lambda \la n\ra) =
\begin{cases}
1 & \text{if $w\lambda \prec \lambda$ and $n = 2(\delta_\lambda - \delta_{w\lambda})$,} \\
0 & \text{otherwise.}
\end{cases}
\]
Then prove that the image of any nonzero map $A_{w\lambda} \to A_{\lambda}\la 2(\delta_\lambda - \delta_{w\lambda}) \ra$ contains $\fIC_{\dom(\lambda)}\la \delta^*_{\dom(\lambda)} - 2\delta_{w\lambda} \ra$ as a composition factor.

\section*{Acknowledgements}

I am grateful to Chris Dodd, Carl Mautner, and Simon Riche for numerous helpful comments and suggestions on an earlier draft of this paper.


\end{document}